\documentclass[a4paper,11pt,reqno]{amsart}

\usepackage{enumerate,amsmath,amsbsy,amsfonts,amsthm,amssymb,wasysym,MnSymbol} 
\usepackage{graphics, graphicx, epsfig, wrapfig} 
\usepackage[retainorgcmds]{IEEEtrantools} 

\usepackage{color} 

\newcommand{\bbR}{\mathbb R}

\newcommand{\bp}{\bar\phi}
\newcommand{\wc}{{\rm w}_c}

\newcommand{\bbP}{\mathbb P}
\newcommand{\bbZ}{\mathbb Z}

\renewcommand{\leq}{\leqslant}
\newcommand{\inte}{\stackrel\circ}

\newcommand{\abs}[1]{ \lvert #1 \rvert}
\newcommand{\tmix}{T_{\rm mix}}

\newcommand{\tri}{\mathbb T}

\newtheorem{theorem}{Theorem}[section]
\newtheorem{conjecture}{Conjecture}[section]
\newtheorem{maintheorem}{Theorem}

\newtheorem{lemma}[theorem]{Lemma}
\newtheorem{proposition}[theorem]{Proposition}
\newtheorem{corollary}[theorem]{Corollary}
\newtheorem{remark}[theorem]{Remark}

\newtheorem{claim}[theorem]{Claim}

\newtheorem{definition}[theorem]{Definition}

\title[Lozenge dynamics and macroscopic shape]{Lozenge tilings, Glauber dynamics and macroscopic shape}
\author{Beno\^it  Laslier}
\address{
 Institut Camille Jordan, Universit\'e Lyon 1, 43
  bd du 11 novembre 1918, 69622 Villeurbanne, France\\
E-mail: laslier@math.univ-lyon1.fr}
\author{Fabio Lucio Toninelli}
\address{CNRS and 
 Institut Camille Jordan, Universit\'e Lyon 1, 43
  bd du 11 novembre 1918, 69622 Villeurbanne, France\\
E-mail: toninelli@math.univ-lyon1.fr}

\begin{document}

\maketitle

\begin{abstract}
  We study the Glauber dynamics on the set of tilings of a finite
  domain of the plane with lozenges of side $1/L$. Under
  the invariant measure of the process (the uniform measure over all
  tilings), it is well known
  \cite{CKP} that the random  height function associated to the tiling
  converges in  probability, in the scaling limit
  $L\to\infty$, to a non-trivial macroscopic shape minimizing a
  certain surface tension functional. According to the boundary
  conditions the macroscopic shape can be either analytic or contain
  ``frozen regions'' (Arctic Circle phenomenon \cite{CLP,JPS}).

  It is widely conjectured, on the basis of theoretical
  considerations \cite{Spohn,Henley}, partial mathematical results
  \cite{Wilson,CMT} and numerical simulations for similar models (\cite{Destainville}, cf. also the bibliography in \cite{Wilson,Henley}), that the Glauber
  dynamics approaches the equilibrium macroscopic shape  in a time of order $L^{2+o(1)}$.
  In this work we prove this conjecture, under the assumption
  that the macroscopic equilibrium shape contains no ``frozen
  region''. 

\end{abstract}

\section{Introduction}

Random lozenge tilings and their Glauber dynamics are a very natural
object in mathematical physics, probability, combinatorics and
theoretical computer science.  Let $\mathcal T_L$ be the triangular
lattice of mesh $1/L$ and call the union of two adjacent triangular
faces a ``lozenge''. A region of $\mathcal T_L$ is called tileable if
it can be covered by non-overlapping lozenges, so that no hole is
left, cf. Figure \ref{fig:lozenges}. Typically, the number of possible
tilings of a tileable region grows like the exponential of $L^2$ time
its area, when the lattice mesh tends to zero. To a lozenge tiling is
naturally associated a height function, so that a tiling can be seen
as a discrete interface, see again Figure \ref{fig:lozenges}.  When
the mesh tends to zero and the height function at the boundary of the
domain tends to some well-defined boundary height $\varphi$, the
height function of a random tiling sampled from the uniform measure
tends in probability to a certain limit shape $\bp$. This limit shape
minimizes the surface energy functional defined in formula
\eqref{eq:functional}, compatibly with the boundary height. According
to the choice of the boundary height, $\bp$ is either analytic, with
$\nabla\bp$ contained in the interior of a bounded set $\tri$ of
``allowed slopes'' ($\tri$ is a triangle, see later) or it can show
coexistence of analytic portions (``liquid phase'') and ``frozen
regions'' or facets where $\nabla\bp$ is on the boundary of $\tri$
(facets correspond microscopically to regions where at least one of
the three types of lozenges has vanishing probability of
being present). For special boundary heights, $\bp$ can happen to be
non-frozen and flat (with constant slope in the interior of $\tri$).

The Glauber dynamics on lozenge tilings is a natural Markov process
whose updates consist in rotating by an angle $180^\circ$ three lozenges that share a
vertex, see Figure \ref{fig:flip}.
\begin{figure}[h]
  \includegraphics[width=5cm]{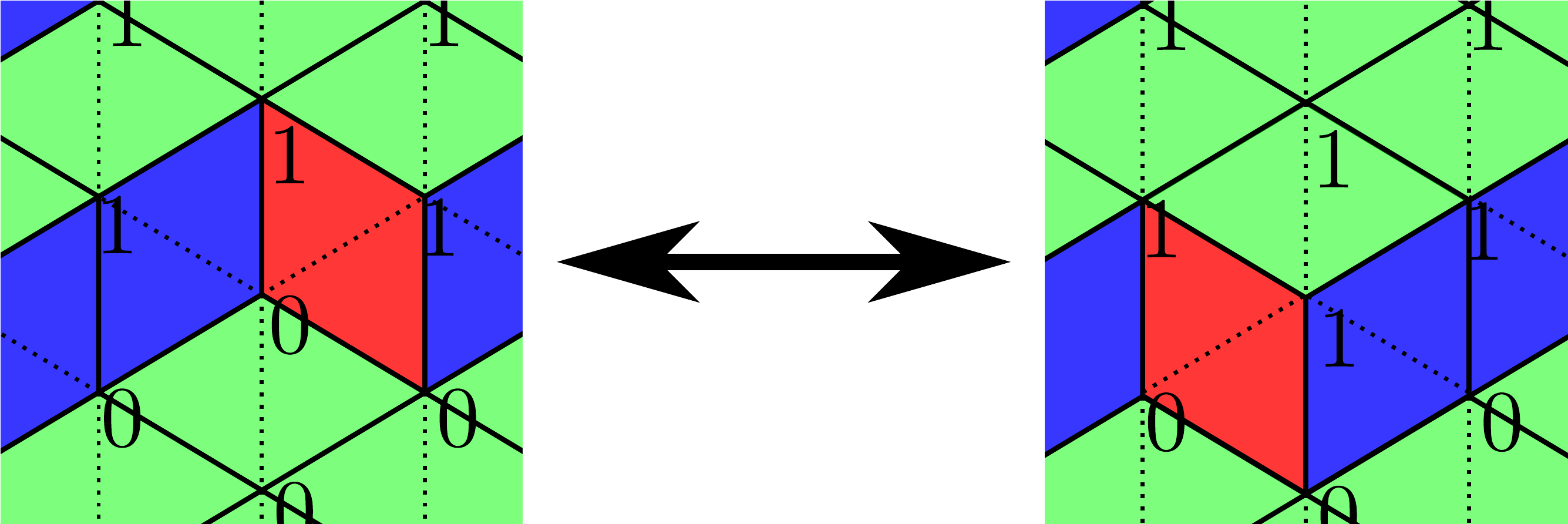}
\caption{The allowed updates. }
\label{fig:flip}
\end{figure}
Such dynamics received a lot of attention in theoretical computer
science \cite{LRS,Wilson,RT} since it is a conceptually and
algorithmically simple way of sampling a random tiling in the
long-time limit (the invariant measure of the process is the uniform
one)\footnote{Let us mention that there are alternative,
  algorithmically more efficient,
  ways  to sample uniform random tilings, see for instance
  \cite{Wil2,KW} or \cite{Mucha}.
}. In this context, a natural question that was investigated in the
mentioned works was, how long one should run the dynamics before the
uniform measure is reached. The Glauber dynamics is an even more natural
stochastic process in mathematical physics: In the height
function representation, lozenge dynamics is equivalent to the
zero-temperature dynamics of interfaces, separating ``$+$'' and
``$-$'' spins, for the \emph{three-dimensional} Ising model. Then, the
question of convergence to equilibrium takes a rather different
flavor: how long does it take before the interface, started far from
equilibrium, approximates the macroscopic shape? Does the stochastic
evolution converge to a deterministic, macroscopic evolution under
suitable time rescaling?

It is widely believed that the time to reach equilibrium should scale
like $L^2$: actually, one expects that under diffusive scaling of time
(i.e. setting $\tau=t/L^2 $) the limiting deterministic evolution  of
the height function $\phi$ should
roughly be the gradient flow associated to the surface energy
functional, 
\begin{eqnarray}
  \label{eq:5}
  \frac{d}{d\tau}\phi= \mu(\nabla\phi)\mathcal L\phi. 
\end{eqnarray}
Here
 $\mathcal L$, directly related to the first variation of the surface energy functional, is  the non-linear elliptic operator defined in
\eqref{eq:PDE}, while $\mu(\nabla\phi)$ is a ``mobility coefficient''. See
\cite{Spohn} for an illuminating discussion of these issues.
For $\tau\gg1$ (i.e. $t\gg L^2$) the interface should asymptotically reach the
macroscopic shape, characterized by $\mathcal L\bp=0$.
This belief is supported  by numerical
simulations (for this and related models, see \cite{Destainville} and
references in \cite{Wilson,Henley}), heuristic arguments \cite{Henley}
and partial mathematical results \cite{Wilson,CMT}.

Let us also mention that, for the zero-temperature \emph{two-dimensional}
(and not three-dimensional)
Ising model, convergence of the evolution of
spin droplets to a deterministic equation of anisotropic mean-curvature
type under diffusive scaling has been
achieved very recently \cite{LST,LST2}. The limit equation is
somewhat the analog of \eqref{eq:5}, with the notable difference that
in that case
$\phi$ describes a curve in the plane and not a surface in
three-dimensional space. What helps in the
two-dimensional case is that, as observed in \cite{Spohn}, the
stochastic interface evolution can be locally mapped to well-studied
interacting particle
processes like one-dimensional symmetric simple exclusion and zero-range
processes. None of these mappings  holds in the three-dimensional case
(i.e. for lozenge dynamics) and a host of new ideas is called for.

\medskip

The main result of the present work  is a mathematical confirmation of the $L^2$
scaling of the equilibration time:
\begin{maintheorem}[Informal version]\label{th:inf1}
  If the macroscopic shape $\bp$ contains no frozen region then, whatever the initial condition of the dynamics, at time $L^{2+o(1)}$ the height function is 
with high probability
at distance $o(1)$ from $\bp$.
\end{maintheorem}
See Theorem \ref{th:pratica2} for a precise formulation. Some previous
results in this direction are mentioned in Section \ref{sec:previous}.

It is at present unclear to us whether the restriction to non-frozen
macroscopic shapes is just a technical limitation or if something
deeper happens. Frozen regions reflect the singularities of the
surface tension functional and it is a priori possible
 that such singularities might have a drastic effect on
dynamics. For the two-dimensional Ising model at zero temperature,
singularities of the surface tension do not modify the time scaling
$L^2$, but they have the effect that the deterministic macroscopic interface evolution one
obtains in the diffusive limit  is not smooth:
the curvature of the interface is in general not differentiable in space \cite[Section 2.2.2]{LST}.

\medskip

An obvious difficulty in attacking the lozenge dynamics problem is
that we have no
\emph{a priori} knowledge of the non-equilibrium interface fluctuations
during the evolution (before the equilibrium state is reached), or even of their order of
magnitude. A natural idea is to 
 look at the system on mesoscopic regions, sufficiently small so that
 macroscopic properties of the interface (slope, curvature, ...) are
almost constant but much
larger than the lattice spacing so that statistical fluctuations are small.
More precisely one might expect that, if at some time $t$ the interface
approximates some smooth height function $\phi_t$, then locally in the
neighborhood of a point where the normal vector to $\phi_t$ is ${\bf
  n}$, the statistics of the interface will be determined by the
infinite-volume, translation invariant Gibbs state of slope ${\bf
  n}$ \cite{KOS}.
This reasoning would suggest height fluctuations of order
$\sqrt{\log L}/L$. ``Local equilibrium'' ideas of this type
are rather
classical in physics, for the macroscopic derivation of the equations
of fluid dynamics
from microscopic particle systems.

In our case, this intuition seems extremely difficult to substantiate
mathematically, yet we do use it somehow. Indeed the route we follow to prove Theorem \ref{th:inf1} is to 
show  that, if time is rescaled a bit more than diffusively (by setting
$\tau=t/L^{2+\epsilon}$ with
$\epsilon>0$ arbitrarily small), then the interface is bounded above
and below by two evolving
 surfaces that follow an auxiliary deterministic equation that morally looks
like \eqref{eq:5} and both converge to the macroscopic shape.
 Via  this auxiliary, slowed down, evolution we are able
to make use of the ``local equilibrium'' intuition
mentioned above.


We will not try to explain in detail the idea of the proof of Theorem
\ref{th:inf1} in this introduction.
At this stage, let us just point out that one of the main difficulties
we have to overcome is to precisely estimate average height and
height fluctuations in mesoscopic regions of size $L^{-1/2+\delta}$ with
$\delta$ small, for a rather large class of boundary heights.
In \cite{CMT} the analog of Theorem \ref{th:inf1} in the special case
where the macroscopic shape is flat (i.e. $\bp$ is an affine function)
was proven: there, the problem of controlling equilibrium fluctuations
in finite domains was bypassed since it was possible to reduce
to fluctuation estimates in the translation invariant infinite-volume
Gibbs states mentioned above. In the present
case, this trick cannot work.

In general, it is only for special domains and
boundary conditions that precise estimates on height fluctuations and
on the finite-$L$ corrections to the average height w.r.t. to the
macroscopic limit $\bp$ are known in the literature. See for instance
\cite{Kdom} for domino tilings. More relevant for us are the works
\cite{Petrov1,Petrov2} by L. Petrov. There, the author
considers uniform random tilings of a hexagon $\varhexagon_{abc}$ of
sides $a,b,c$. In this case, the macroscopic shape (that is not flat)
can be written down ``explicitly'' \cite{CLP} (cf. also Sections
\ref{sec:hexa} and \ref{sec:localstructures} below) and the equilibrium measure has an exact determinantal
representation \cite{Petrov1,Petrov2} which allows for precise
asymptotic analysis.  

One of the main new ideas of our work is that, locally in regions of
size $\approx L^{-1/2+\delta}$, we can compare the height of the
randomly evolving interface with the random \emph{equilibrium} height
of lozenge tilings in a hexagonal region $\varhexagon_{abc}$ with
suitably chosen, time-dependent parameters $a,b,c$.  Technically, one
key result we prove, which might be of interest by itself, is the
following:
\begin{maintheorem}[Informal version]
\label{th:inf2}
  Given a macroscopic shape $\bp$ in a domain $U$, if $\bp$ is smooth in the neighborhood of a point $u\in U$, then the Taylor expansion of $\bp$ around $u$ coincides  up to second order with the Taylor expansion of the macroscopic shape associated to some hexagon $\varhexagon_{abc}$.
\end{maintheorem}
We will call the second-order Taylor expansion of $\bp$ at a given
point a ``local structure''.  We would like to emphasize that Theorem
\ref{th:inf2} is \emph{a priori} not obvious:  As we will see in Section
\ref{sec:localstructures}, the set of all admissible local structures
associated to arbitrary macroscopic shapes is parametrized by four
variables (two for the slope and two for the Hessian matrix), while
``hexagonal'' local structures are parametrized by a different set of
four variables with a rather different meaning (two for the
side-lengths of the hexagon and two for the coordinates of a point
inside the hexagon). We have then to check that a certain explicit but complicated
function from $\bbR^4 $ to $\bbR^4$ is surjective (actually it turns
out to be a bijection).

Theorem \ref{th:inf2} would be false
if ``second order'' were replaced, say, by ``third order'' (it would
require surjectivity of a function from $\bbR^4$ to $\bbR^8$).  Remarkably, for the proof of Theorem \ref{th:inf1} the
second-order comparison provided by Theorem \ref{th:inf2} is
sufficient. The basic reason is that, in regions of size
$L^{-1/2+\delta}$, third- or higher-order terms in the expansion of
the macroscopic shape give negligible contributions of order
$L^{-3/2+3\delta}$, much smaller than the minimal significant
length-scale of the model, which is the lattice spacing $1/L$.

\section{Random lozenge tilings and height function}

\subsection{Monotone surfaces and height functions} 
\label{sec:height}
Let $L$ be an integer, that will be taken large later. Closed squares
in $\bbR^3$ of side $1/L$, with the four vertices
in $(\bbZ/L)^3$, will be called \emph{ faces of
$(\bbZ/L)^3$}.
\begin{definition}
A discrete (or stepped) monotone surface $\Sigma_L$ is a 
connected union of faces of $(\bbZ/L)^3$ that projects
bijectively on the $111$ plane $P_{111}$.
\end{definition}

Look at Figure \ref{fig:lozenges}: the projection of each square
face of $\Sigma_L$ is a lozenge with angles $60^\circ$ and $120^\circ$
and three possible orientations: horizontal, south-east
and south-west, according to whether the normal vector to the  square face of $\Sigma_L$ is
$(0,0,1)$, $(1,0,0)$ or $(0,1,0)$. 
The projection of $\Sigma_L$ gives therefore a lozenge tiling of $P_{111}$ with these
three types of tiles. 
Vertices of
the lozenges are the vertices of a triangular lattice $\mathcal
T_L$ of side $const./L$. Via a suitable choice of coordinates,
we will set the constant to be $1$ below.

\begin{definition}
A continuous monotone surface $\Sigma$ in $\bbR^3$ is a
two-dimensional connected
surface such that:
\begin{enumerate}
\item 
$\Sigma$ projects bijectively on  $P_{111}$;
\item   the normal vector to $\Sigma$, assumed to be
defined almost everywhere, points in $\bbR^3_{\ge 0}$. 
\end{enumerate}
\end{definition}
Note that injectivity of the orthogonal projection of $\Sigma$ on $P_{111}$ is
 a consequence of the assumption on the normal vector.

\begin{definition}[Height function] To a continuous (resp. stepped)
  monotone surface $\Sigma$ (resp. $\Sigma_L$) we associate a height
  function $\phi:P_{111}\to \mathbb R$ (resp. $h: \mathcal T_L\to
  \bbZ/L$
)
  , as follows: $\phi(u)$ (resp. $h(u)$) equals the height
  with respect to the horizontal plane of the point $p\in\Sigma$
  (resp. $p\in\Sigma_L$) whose orthogonal projection on $P_{111}$ is $u$.
\end{definition}
Note that, for discrete monotone surfaces, heights are associated to
vertices of lozenges, i.e. to vertices of $\mathcal T_L$. The
definition can be extended to obtain a real-valued height function $h$ on the whole
$P_{111}$, simply by establishing that the height is linear on
triangular faces of $\mathcal T_L$.

\begin{figure}
  \includegraphics[width=10cm]{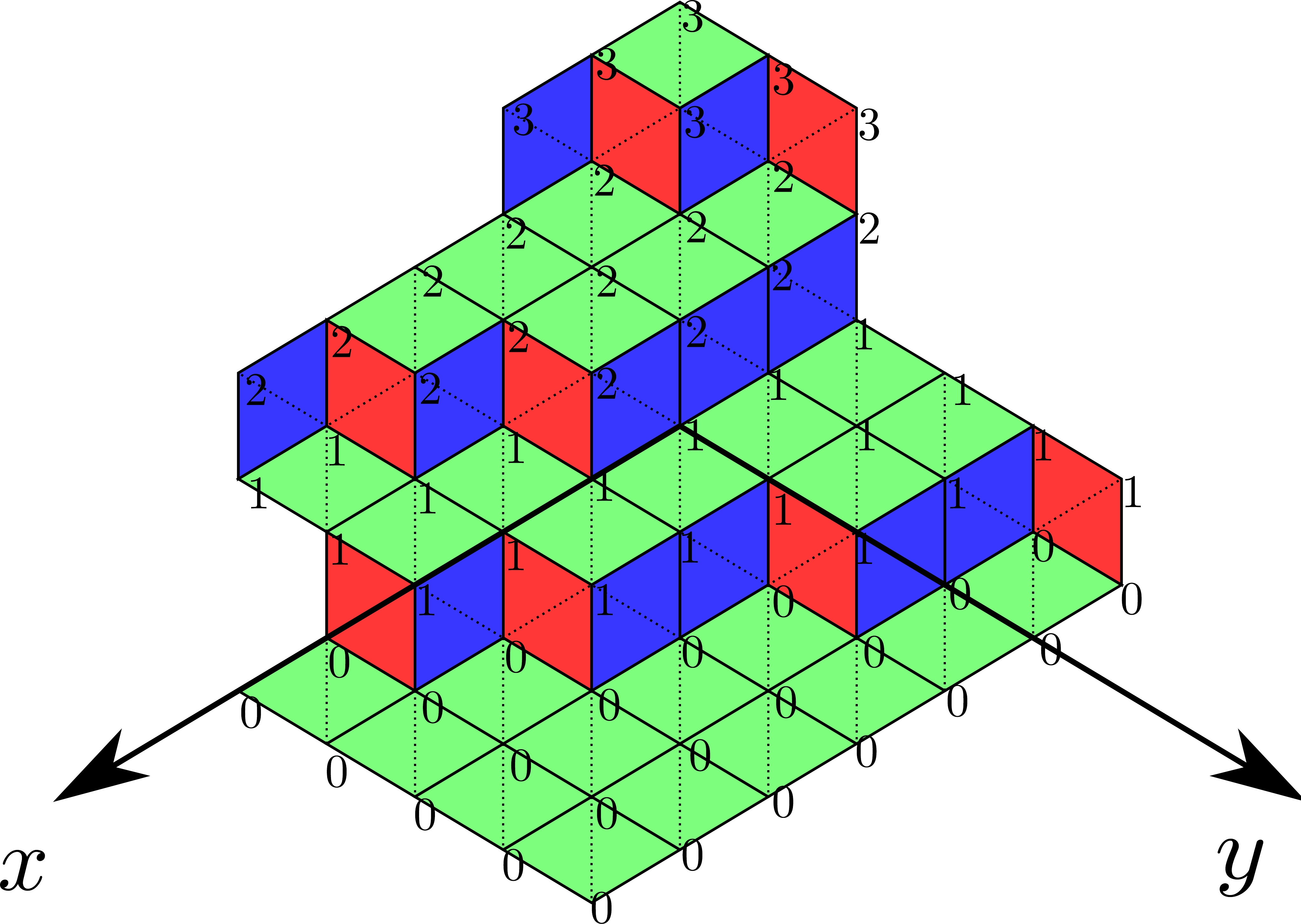}
\caption{A portion of lozenge tiling associated to a stepped monotone
  interface. Next to each vertex of $\mathcal T_L$ is marked  the
  value of the height function. The underlying triangular lattice is given by dotted lines.}
\label{fig:lozenges}
\end{figure}

\smallskip

On $P_{111}$
introduce a coordinate frame $(x,y)$ (see Figure \ref{fig:lozenges}) such that a
given reference vertex $v_0\in\mathcal T_L$ has coordinates $(0,0)$ and the
vertices of $\mathcal T_L$ that are nearest neighbors of $v_0$ 
in directions $e^{-5i\pi/6}$ (resp. $e^{-i\pi/6}$) have
coordinates $(1/L,0)$ (resp. $(0,1/L)$).
This choice of coordinates is convenient for stepped
monotone surfaces, since the axes are along two of the directions of
the triangular lattice $\mathcal T_L$. 
Note  also that the $x$ and
$y$ axes are along the  $P_{111}$ projections of the directions of the
usual $\hat e_1$ and $\hat e_2$ coordinate axes of $\bbR^3$.
Whenever convenient, we will implicitly identify the plane $P_{111}$
with $\bbR^2$.

For continuous monotone interfaces, the condition that the normal vector points
in $\bbR^3_{\ge 0}$ can be reformulated as follows: wherever defined, the gradient $\nabla
\phi=(\partial_x \phi,\partial_y \phi)$ belongs to $\mathbb T$, where:
\begin{definition}
\label{def:tri}
$\mathbb T\subset \bbR^2$ is the triangle with vertices $(0,0),(0,-1),(-1,0)$.
\end{definition}

\begin{remark}
\label{rem:hl}
When one moves
by one lattice step in $\mathcal T_L$ along the $x$ or $y$ directions 
the height function of a stepped interface decreases by $1/L$ if one
crosses a lozenge, and is unchanged if one moves along the edge of a
lozenge.  When instead one moves by a lattice step upward in the vertical direction (i.e. by $(-1/L,-1/L)$ in the $(x,y)$ coordinates),
the height function is unchanged if one
crosses a lozenge, and increases by $1/L$ if one moves along the edge of a
lozenge.  
\end{remark}
While a discrete height function uniquely identifies a
lozenge tiling, in view of Remark \ref{rem:hl} a lozenge tiling identifies the height
function $h$ only modulo a global additive constant (the tiling
identifies the height gradients). If however one
fixes the height at some vertex, then the correspondence is bijective. In the following,
the height along the boundary of a finite region will be fixed, so we will freely identify height
functions and lozenge tilings.

\begin{definition}[Domains] \label{def:domains} In the continuous surface setting, a
  domain $U$ will denote a simply connected, bounded, closed subset of
  $P_{111}$, whose boundary $\partial U$ is a piecewise $C^\infty$
  simple curve. 

  In the discrete setting, a discrete domain $U_L$ will be a simply
  connected, bounded union of closed triangular faces of $\mathcal
  T_L$. With some abuse of notation, we will often identify $U_L$ with
  $U_L\cap\mathcal T_L$.  The set of sites $v\in U_L\cap\mathcal T_L$
  that are not on $\partial U_L$ is denoted $U_L^{int}$, while
  $U^{ext}_L=\mathcal T_L\setminus U^{int}_L$.

\end{definition}

\begin{definition}[Non-extremal monotone surface]\label{def:strictly}
  Let $\Sigma$ be a continuous monotone surface, with height function
  $\phi$, and $U$ be a domain of $P_{111}$. We say that $\Sigma$ is 
  non-extremal in $U$ if
  $\nabla \phi$ is defined everywhere  in $\stackrel \circ U$ (the
  interior of $U$) and   there exists 
$\epsilon>0$ such that,  for every $(x,y)\in \stackrel \circ U$, $\nabla \phi$  is at
distance at least $\epsilon$ from the boundary of the triangle
$\mathbb T$.
\end{definition}
In geometric terms, this means that all three components of the normal vector to $\Sigma$ 
are larger than a constant times $\epsilon$, at
every point that projects on $\stackrel \circ U$.

\begin{definition}[Continuous boundary heights] 
\label{def:cbh}
Given a domain $U\subset P_{111}$,
a function  $\varphi: P_{111}\setminus U\mapsto 
 \mathbb R$ is called \emph{(continuous) boundary height} if there exists
a continuous monotone surface $\Sigma$ whose
height function $\phi$ coincides with $\varphi$ on $P_{111}\setminus U$.
\end{definition}
Discrete boundary heights are defined similarly:
\begin{definition}[Discrete boundary heights]
\label{def:dbh}
Given a discrete domain $U_L$ as in Definition \ref{def:domains}, 
  we call $\varphi_L:U_L^{ext}\mapsto \bbZ/L$ a \emph{discrete boundary height}
if there exists a stepped monotone surface $\Sigma_L$ whose height coincides
with $\varphi_L$ on 
$ U_L^{ext} $.
\end{definition}

\begin{remark}
Boundary heights have been defined for technical reasons as height functions
outside certain (continuous or discrete) domains.  However, with some abuse of notation, we will often  see
$\varphi_L$ and $\varphi$ as functions on $\partial U_L$ and $\partial
U$, respectively (instead of functions on $U^{ext}$ and
$P_{111}\setminus U$). This makes sense because we will see (DLR
equations below) that the statistical properties of the height
function in a domain $U_L$ are determined uniquely by the height on
$\partial U_L$.

\end{remark}

\smallskip

We will be mostly interested in stepped monotone surfaces that approximate
as $L\to\infty$ a continuous monotone surface:
\begin{definition}[Discretizations] Consider a continuous monotone
  surface $\Sigma$, a domain $U$ and, for $L\ge1$, stepped
  monotone surfaces $\Sigma_L$ and discrete domains $U_L$. We say that $(\Sigma_L,U_L)_{L\ge1}$ is a
  discretization of $(\Sigma,U)$ if, for some constant $C$ independent
  of $L$:
  \begin{enumerate}
  \item the boundary $\partial U_L$ is within Hausdorff distance
  $C/L$ from $\partial U$;
\item  for every $u\in U_L\cap U $,
  one has $|h(u)-\phi(u)|\le C/L$. 
  \end{enumerate}
The restriction $\varphi_L$ of $h$ (the height function of $\Sigma_L$)
  to $ U_L^{ext}$ is said to be a discretization of the boundary
  height $\varphi=\phi|_{P_{111}\setminus  U}$.
\end{definition}
Given $(\Sigma,U)$, one can always
find a discretization $(\Sigma_L,U_L)_{L\ge1}$: just take $\Sigma_L$
as the boundary of the
union of all closed cubes with vertices in $(\bbZ/L)^3$ that are below $\Sigma$, and $U_L$
as the union of triangular faces of $\mathcal T_L$ contained in $U$.

\subsection{Uniform measure, DLR equations and macroscopic shape}

\label{sec:macro}

Given a discrete domain $U_L$ and a discrete boundary height
$\varphi_L$ as in Definition \ref{def:dbh}, we let
$\pi_{U_L}^{\varphi_L}$ denote the uniform measure over the set 
$\Omega_{U_L,\varphi_L}$ of all stepped
monotone surfaces whose height on
$ U_L^{ext}$ is $\varphi_L$ (by definition, there is at least one of
them). 

The measure $\pi_{U_L}^{\varphi_L}$ satisfies the so-called DLR
equations. If $V_L$ is a sub-domain of $ U_L$, then under the
law $\pi_{U_L}^{\varphi_L} $, conditioned to the event that the
height on $ V_L^{ext}$ is a certain boundary height $\psi_L$, the
height function in $V_L$ has the uniform law
$\pi_{V_L}^{\psi_L}$.

The following well-known theorem states that, if the boundary
condition $\varphi_L$ is the discretization of a continuous boundary
height $\varphi$, with
high probability under the uniform measure $\pi_{U_L}^{\varphi_L}$ the stepped interface
$\Sigma_L$ approximates a certain macroscopic 
shape $\bar\phi$, that
solves a variational principle.
\begin{theorem}\cite{CKP}
\label{th:macro}
Let $U$ and $\varphi$ be a domain and a continuous
boundary height, satisfying the properties specified in Definitions 
\ref{def:domains} and \ref{def:cbh}.
  \begin{enumerate}
  \item   There exists a unique minimizer $\bar \phi$, 
 among continuous monotone surfaces with boundary height $\varphi$,
 of the surface tension functional
\begin{eqnarray}
  \label{eq:functional}
  \Psi(\phi)=\int_U \sigma(s(u),t(u))d^2u=\int_U \sigma(\nabla \phi)d^2u
\end{eqnarray}
where
\begin{eqnarray}
  \label{eq:sigma}
  \sigma(s,t)=-\frac1\pi\left[
\Lambda(\pi s)+\Lambda(\pi t)+\Lambda(\pi(-1-s-t))
\right]
\end{eqnarray}
and
\[
\Lambda(\theta)=-\int_0^{-\theta} \log (2 \sin(t))dt
\]
(observe that $-\theta\ge 0$, since $s,t,-1-s-t\le0$ if $(s,t)\in \tri$).

\item Let $(\varphi_L, U_L)_{L\ge1}$ be a discretization of the boundary condition 
$(\varphi, U)$ and 
$\Sigma_L$ be distributed according to the uniform measure $\pi^{\varphi_L}_{U_L}$. Then, as $L\to\infty$, 
$\Sigma_L$ tends in $\pi_{U_L}^{\varphi_L}$-probability to $\bar \phi$: for every
$\epsilon>0$, 
\[
\pi^{\varphi_L}_{U_L}\left(
\exists u\in U_L: |h(u)-\bar\phi(u)|\ge \epsilon
\right)\to0.
\]
  \end{enumerate}
\end{theorem}

The Euler-Lagrange equation associated to the variational principle is
 the
non-linear elliptic PDE in ``divergence form'' 
\begin{multline}
  \label{eq:PDE}
\mathcal L\bp:=
\partial_x(\partial_s\sigma(\nabla \bar \phi))+\partial_y(\partial_t\sigma(\nabla\bar\phi))
\\
=a_{11}(\nabla \bar\phi)\partial^2_x\bar\phi+a_{22}(\nabla \bar\phi)\partial^2_y\bar\phi
+2a_{12}(\nabla \bar\phi)\partial^2_{x,y}\bar\phi=0
\end{multline}
with 
\begin{align}
\label{eq:aij}
  a_{11}(\nabla \bar\phi)= \frac1{\tan(-\pi\partial_x\bar\phi)}+\frac1{
\tan(\pi(1+\partial_x\bar\phi+\partial_y\bar\phi))}\\
 a_{22}(\nabla \bar\phi)= \frac1{\tan(-\pi\partial_y\bar\phi)}+\frac1{
\tan(\pi(1+\partial_x\bar\phi+\partial_y\bar\phi))}\\
a_{12}(\nabla \bar\phi)=a_{21}(\nabla \bar\phi)=
\frac1{
\tan(\pi(1+\partial_x\bar\phi+\partial_y\bar\phi))}.
\end{align}
The matrix ${\bf a}(s,t)=\{a_{ij}((s,t)) \}_{i,j=1,2}$ is strictly positive
definite in $\stackrel \circ\tri$, as a consequence of strict convexity
of the surface tension functional $\Psi$ (positive definiteness can
also be checked by hand; in particular, the determinant of ${\bf a}(s,t)$ is $1$). In $\stackrel\circ\tri$ the
matrix elements  $a_{ij}((s,t))$ are analytic and the diagonal
elements $a_{ii}(s,t)$ are strictly positive. When instead $(s,t)$
approaches $\partial\tri$, the matrix ${\bf a}(s,t)$ becomes singular.

Assume that $\bp$ is non-extremal in $U$ (its gradient is bounded away from the boundary of the set of allowed slopes). Then, $\bp$
is real analytic in $\stackrel\circ U$
(see for instance  \cite[Ch. II.2 and Ch. VI.3]{Giaquinta}) and solves \eqref{eq:PDE} everywhere in $\stackrel\circ
U$.

It can however happen, even for some natural boundary conditions
$(U,\varphi)$ (see Section \ref{sec:hexa}), that in some subset
$\hat U\subset U$ with non-empty  interior the gradient $\nabla \bp$ belongs to
$\partial \tri$. Such regions $\hat U$ are called \emph{frozen regions}.

\section{Dynamics, conjectures and main result}

The Glauber dynamics is defined as a Markov process
$(h^\eta_t)_{t\ge0}$ on the set $\Omega_{U_L,\varphi_L}$, with $\eta$
denoting the initial condition.  To each site $v\in U_L^{int}$ such that all six neighbors of $v$ are in $U_L$, we associate a
mean-one Poisson clock. Clocks at different sites are
independent. When the clock at $v$ rings, if in the present lozenge
configuration $v$ belongs to 
exactly three lozenges, then we turn the three lozenges by an angle
$\pi$. See Figure \ref{fig:flip}.
In terms of height function, an update
corresponds to increasing by $+1/L$ or decreasing by $-1/L$ the height
$h^\eta_t(v)$ with rate $1$, with the constraint that $h^\eta_t$ remains a
stepped monotone surface in $\Omega_{U_L,\varphi_L}$ at all times.

We denote $\mu_t^\eta$ the law of $h_t^\eta$ and $\bbP$ the law of the
entire process.
The dynamics is reversible and its unique invariant measure is the
uniform measure $\pi_{U_L}^{\varphi_L}$.

As we mentioned in the introduction, the dynamics is expected to converge to equilibrium in
a time of order $L^2$ times sub-leading corrections. 
More precisely:
\begin{conjecture}
\label{th:idealmente}
Let   $U $, $\varphi$ and its discretizations
$(U_L,\varphi_L)_{L\ge1}$ be as above.
  For every $\delta>0$ 
  there exists $c(\delta)<\infty$ such that, whatever
  the initial condition $\eta$, at times
  $t>c(\delta) L^{2+\delta}$ the following holds with  probability tending to $1$ as $L\to\infty$:
for every vertex $v\in U_L$
\begin{eqnarray}
  \label{eq:6}
|h^\eta_t(v)-\bar \phi(v)|=o(1).
\end{eqnarray}

\end{conjecture}
In other words, within time $L^{2+o(1)}$ the interface
macroscopically approximates
the equilibrium shape to any pre-assigned precision.

Actually, we believe that more should be true: at time
$L^{2+o(1)}$, the law $\mu_t^\eta$ should be very close to the
equilibrium measure  $\pi_{U_L}^{\varphi_L}$.
More precisely, define  the mixing time of the dynamics as
\begin{eqnarray}
  \label{eq:17}
  \tmix=\tmix(U_L,\varphi_L)=\inf\{t:\max_\eta\|\mu^\eta_t-\pi_{U_L}^{\varphi_L}\|\le 1/(2e)\},
\end{eqnarray}
with $\|\mu-\nu\|$ the total variation distance between two
probability measures $\mu,\nu$.
Then:
 \begin{conjecture}
 \label{th:idealmente2}
   In the same setting of Conjecture \ref{th:idealmente}, it is
expected that
   $\tmix=O(L^{2+o(1)})$.
 \end{conjecture}

Thanks to the classical inequality
\begin{eqnarray}
  \label{eq:18}
\max_\eta  \|\mu_t^\eta-\pi_{U_L}^{\varphi_L}\|\le e^{-\lfloor t/\tmix\rfloor},
\end{eqnarray}
this would say that, at time of order $L^{2+o(1)}\log (1/\delta)$,
$\mu_t^\eta$ is within variation distance $\delta$ from equilibrium, for
any arbitrary $\delta$.

\medskip

Our main results is a proof of Conjecture \ref{th:idealmente}  under
the assumption that the macroscopic shape has no frozen regions
(Theorem \ref{th:pratica2}).
As mentioned at the end of Section \ref{sec:hexa}, the methods we
develop in this work allow also to prove the stronger Conjecture
 \ref{th:idealmente2} for a rather special class of boundary conditions
 (Theorem \ref{th:praticamente}; details will be given in a
 forthcoming publication).
\begin{theorem} 
\label{th:pratica2} Let the domain $U$ and the boundary condition
$\varphi$ satisfy the assumptions of Definitions \ref{def:domains} and \ref{def:cbh}. Assume in addition that the
associated macroscopic shape $\bar\phi$ is non-extremal in $U$ and let $(\varphi_L, U_L)_{L\ge1}$ be a discretization of 
$(\varphi,U)$. Consider the Glauber dynamics in $U_L$ with boundary height
$\varphi_L$ and initial condition $\eta$.  There exists a sequence $\epsilon_L$ tending to zero and, for each $\delta>0$,  a constant $c(\delta)<\infty$ such that, with $T_L=c(\delta)L^{2+\delta}$,
\begin{eqnarray}
  \label{eq:provaw}
  \max_\eta\sup_{t>T_L} \mathbb P(\exists v\in U_L: |h^\eta_t(v)-\bar \phi(v) |\ge \epsilon_L)\to 0
\quad\text{ as} \quad L\to\infty.
\end{eqnarray}
\end{theorem}
It is important to emphasize that, except for the $\delta$ in the
exponent of $T_L$, this result is optimal. Indeed,
it is known that there exist initial conditions $\eta$
such that,  for times smaller than $a L^2$
with $a>0$ small (how small, depending on the domain $U$ and on the
boundary height $\varphi$), $\max_v |h^\eta_t(v)-\bar \phi(v)|$ is still bounded
away from zero. This is proven in \cite[Section 10]{dominos} in the
special case where the macroscopic shape $\bp$ is flat, but the proof extends with minor
modifications to the case of non-extremal $\bp$ considered here.

\smallskip

A slight generalization of Theorem \ref{th:pratica2} is the following:
\begin{corollary}
\label{cor:quasifrozen}
The same statement as in Theorem \ref{th:pratica2} holds without the
assumption that $\bp$ is non-extremal, if the following holds:
there exists a sequence
 $(\varphi^{(n)})_{n\ge1}$ of  continuous boundary heights
  on $P_{111}\setminus U$ such that:
  \begin{itemize}
  \item $\max_{x\in\partial U}|\varphi(x)-\varphi^{(n)}(x)|$ tends to zero as 
$n\to\infty$;
  \item for every $n$, the macroscopic shape $\bp^{(n)}$ corresponding to 
boundary conditions $(U,\varphi^{(n)})$ is non-extremal.
  \end{itemize}

\end{corollary}
In other words, the claim of Theorem \ref{th:pratica2} holds if $\bp$
can be approximated by a sequence of non-extremal macroscopic shapes.
A typical application is given in Section \ref{rem:ellisse} below.

\subsection{Previous results} 
\label{sec:previous}
The first mathematical estimate we are aware of on the relaxation time
of the Glauber dynamics for lozenge tilings is in the work of Luby,
Randall and Sinclair \cite{LRS}, who proved \emph{rapid mixing}: the
mixing time grows \emph{at most} like some polynomial of the
graph-distance diameter of $U_L$ (i.e.  $\tmix\le L^C$, with our
notations). The estimate they obtained on $C$ was far from the
expected optimal value $2$, but we emphasize that their result
required essentially no conditions on the boundary height (in
particular, the possible presence of frozen regions in the macroscopic
shape played no role at all).  A few years later, D. Wilson
\cite{Wilson} proved the $O(L^2\log L)$ scaling for $\tmix$, but for
an \emph{ad-hoc} modified, highly non-local Markov dynamics introduced
in \cite{LRS}, whose updates can modify the position of  an unbounded number of
lozenges at the same time. Via known comparison arguments for Markov chains
\cite{DSC}, Wilson's result implies again \cite{RT} a non-optimal
polynomial upper bound $\tmix=O(L^6\log L)$ for the mixing time of the
local Glauber dynamics (actually the log factor can be removed by going through the spectral gap of the
non-local chain, see
\cite[Section 5]{Wilson}).

Both \cite{LRS} and \cite{Wilson} are based on clever path-coupling
arguments: the reason why they cannot catch  the right scaling
$\approx L^2$ for the time to approach the equilibrium shape is,
in our opinion, that they do not use any input from the knowledge of the macroscopic shape and of height
fluctuations properties of the equilibrium measure.

A first mathematical confirmation of the $L^{2+o(1)}$ scaling for the
time of convergence to equilibrium came in \cite{CMT}, where it was
proven that $\tmix\approx L^{2+o(1)}$, under the strongly limiting
assumption that the boundary height is such that the macroscopic shape
$\bp$ is flat (i.e. an affine function).
The same result was proven later, with a somewhat different method,
for more general tilings (e.g. domino tilings)
\cite{dominos}. 

Going beyond the flat case, as we do in the present work, requires
many novel
mathematical ideas.

\subsection{Hexagonal regions}
\label{sec:hexa}

A crucial role in our work is played by some special boundary heights, of
``hexagonal type''. This may look at first surprising since the
associated macroscopic shape is not at all non-extremal, in contrast
with the requirements  of
Theorem \ref{th:pratica2}. Such boundary conditions have played an
extremely important role in the understanding of random tilings: in
particular, this is
the first case where the occurrence of frozen region and of the
``arctic circle phenomenon'' was discovered \cite{CLP} (see also the earlier work \cite{JPS} for domino tilings). Also, the
uniform law on lozenge tilings has in this case an explicit
determinantal representation. This allows to extract sharp estimates, as $L\to\infty$,
on height fluctuations, on the finite-size corrections to the average
height with respect to the macroscopic shape $\bp$, and to prove
convergence of height fluctuations to the Gaussian Free Field  \cite{Petrov1,Petrov2}.

\begin{figure}
  \includegraphics[width=10cm]{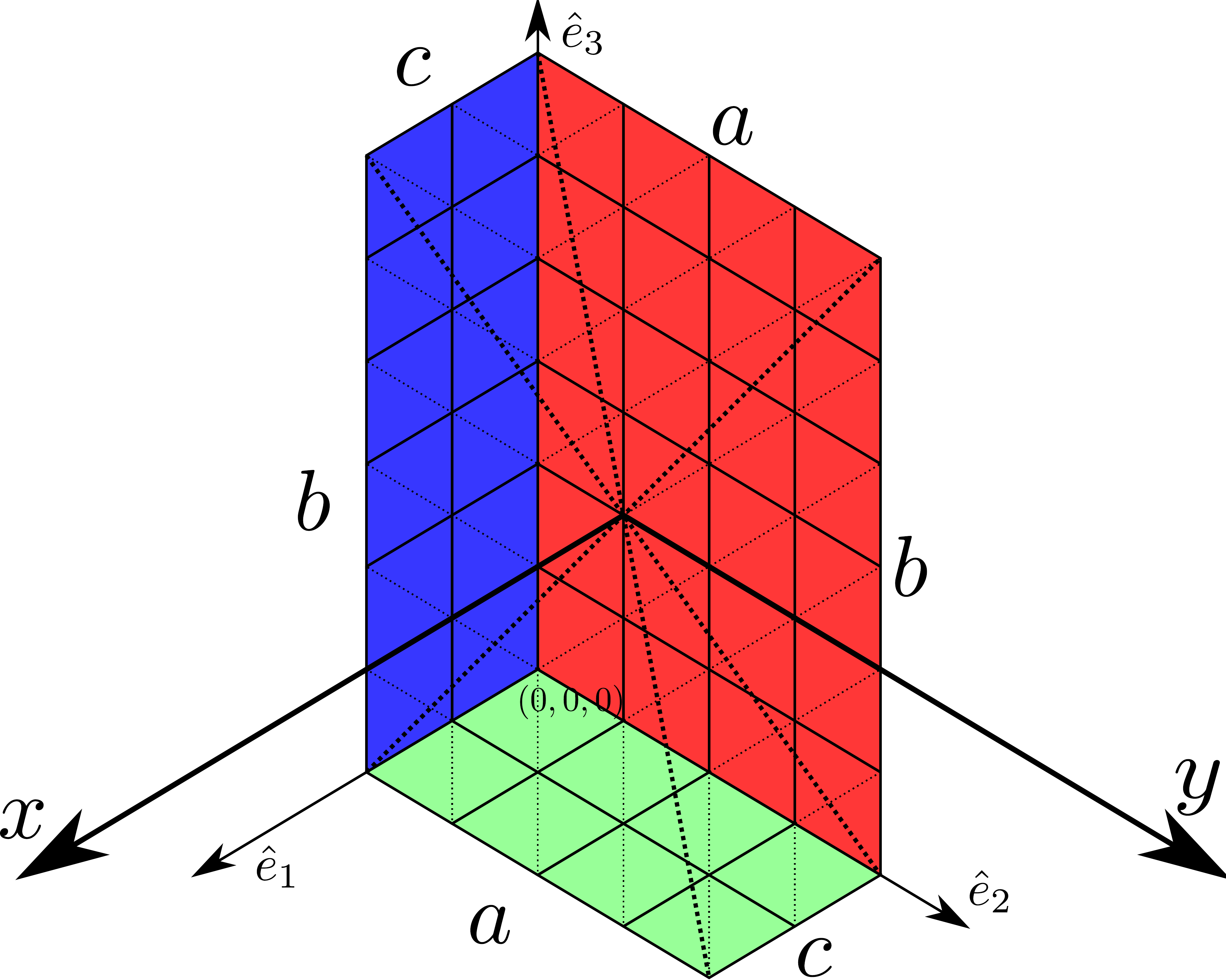}
\caption{The hexagon $\varhexagon_{abc}$. The origin of $(x,y)$ is set
at the center of the hexagon, where the three diagonals (dotted lines)
meet. The lines $\hat e_i$ are the $111$ projections of the positive coordinate
axes of $\bbR^3$.}
\label{fig:hexagone}
\end{figure}

Let the monotone surface $\Sigma$ be the boundary of $(\bbR^+)^3$, the positive
octant of $\bbR^3$ and let the half-infinite lines $\hat e_1,\hat e_2,\hat e_3$
be the $111$  projection of the positive coordinate
axes of $\bbR^3$, see Figure \ref{fig:hexagone}. Given $a,b,c>0$ 
let  $\varhexagon_{abc}$ be the hexagon in $P_{111}$ with angles of $120^\circ$ and with sides $a,b,c,a,b,c$, such that three of the vertices are on the lines $\hat e_i$.  The
sides of length $a$ are parallel to the $y$ axis and  those of length
$c$ to the $x$ axis (recall that the $x$ and $y$ axes are not
orthogonal). See again Figure \ref{fig:hexagone}. 
Without loss of generality, we will assume that $a+b+c=1$.

 We let $\Circle_{abc}$ be the open ellipse
  inscribed in $\varhexagon_{abc}$, and $\varphi_{abc}$ the boundary height of
  $\Sigma$ restricted to $\partial \varhexagon_{abc}$. 

This is the
  prototypical case where the macroscopic shape contains frozen regions:
\begin{theorem}\cite{CLP}\label{th:CLP}
Call $\bar\phi_{abc}$ the macroscopic shape when $U=\varhexagon_{abc}$ and  $\varphi=\varphi_{abc}$. In
$\Circle_{abc}$, $\bar\phi_{abc}$ is
analytic  and its gradient $\nabla \bar\phi_{abc}$ is in
$\stackrel\circ\tri$. On $\varhexagon_{abc}\setminus\Circle_{abc}$, $\nabla
\bar\phi_{abc}\in\partial \tri$. More precisely, remove from $\varhexagon_{abc}\setminus \Circle_{abc}$ the six
points of contact between $\Circle_{abc}$ and the boundary of the
hexagon and consider the six connected components of the set thus obtained. On the components that touch sides $ac$ (resp. $ab$,
resp. $bc$) the gradient is
$(0,0)$ (resp. $(-1,0)$, resp. $(0,-1)$).
\end{theorem}

The ellipse $\Circle_{abc}$ is called the ``smooth region'' or ``liquid
region'', while $\varhexagon_{abc}\setminus \Circle_{abc}$ is the ``frozen region''.

\medskip

A discretization of $(\varhexagon_{abc},\varphi_{abc})$ is simply
obtained as $(\varhexagon_{a_Lb_Lc_L},\varphi_{a_Lb_Lc_L})_{L\ge1}$
where
\[a_L=(1/L)\lfloor a
L\rfloor, b_L=(1/L)\lfloor b L\rfloor, c_L=(1/L)\lfloor c
L\rfloor.\] To simplify formulas we will always pretend that
$aL, bL, cL$ are even integers, in which case $\varhexagon_{a_Lb_Lc_L} $
and $\varphi_{a_Lb_Lc_L}$ exactly coincide with $\varhexagon_{abc},\varphi_{abc}$,
and the center of the hexagon is a vertex of $\mathcal T_L$, that will
be chosen by convention  as the origin of $P_{111}$.
For lightness of notations, we write simply $\pi^{abc}_L$ for the uniform
measure $\pi_{\varhexagon_{a_Lb_Lc_L}}^{\varphi_{a_Lb_Lc_L}}$.

We already know from Theorem \ref{th:macro} that, under the measure $\pi_L^{abc}$,  the typical
height function is macroscopically close to $\bar\phi_{abc}$. The
following theorem makes this claim much sharper, but the statements
hold only in the liquid region:
\begin{theorem}
\label{th:Petrov}
For every $v,u\in  \Circle_{abc}$ 
  \begin{eqnarray}
    \label{eq:ave}
    |(\pi^{abc}_{L}(h(u))-\pi^{abc}_{L}(h(v)))-(\bar\phi_{abc}(u)-\bar\phi_{abc}(v))|\le K |u-v|/L
  \end{eqnarray}
where $K=K(a,b,c,u,v)$ is  bounded as long 
as 
\[\min(a,b,c,dist(u,\partial
\Circle_{abc}),dist(v,\partial \Circle_{abc}))\] is bounded away from zero. 

For every $u\in \Circle_{abc}$, $n>0$ and $\epsilon>0$,
\begin{eqnarray}
  \label{eq:7}
  \pi^{abc}_L\left(|h(u)-\pi^{abc}_L(h(u))|>L^{-1+\epsilon}\right)=O(L^{-n})
\end{eqnarray}
(again, the error term is uniform in $a,b,c,u$ if
$\min(a,b,c,dist(u,\partial
\Circle_{abc}))$ is bounded away from zero). 
\end{theorem}
The first claim is proven in Appendix \ref{app}, following  methods of \cite{Petrov1,Petrov2}.
The second one follows directly from  \cite[Lemma 5.6]{Petrov2}, 
where it is proven that $\pi_L^{abc}(L^n|h(v)-\pi_L^{abc}(h(v))|^n)=O(L^\epsilon)$,
plus Tchebyshev's inequality. 
Uniformity of the error term is
not stated explicitly in \cite{Petrov2}, but it can be easily extracted from the proof).

\smallskip

Let us state and prove a simple consequence of Theorem
\ref{th:Petrov}, that we need in Section \ref{sec:proofeq}. 
\begin{proposition}
\label{th:Petrov2}
Let as above $a,b,c>0$ with $a+b+c=1$ and $D_L$ be a discrete domain
contained in
$\Circle_{abc}$, whose distance from $\partial \Circle_{abc}$ is
at least $\delta>0$ independent of $L$. Let further $\varphi_L$ be a boundary height on $\partial D_L$ such that, for some  $C>0$,
\[
|\varphi_L(v)-\pi_L^{abc}(h(v))|\le C/L \quad\text{for every}\quad v\in \partial D_L.
\]
 Then, for every
$\epsilon>0$ and $n<\infty$,
\begin{eqnarray}
  \label{eq:25}
 \sup_{v\in D_L}
 \pi_{D_L}^{\varphi_L}(|h(v)-\pi_L^{abc}(h(v))|\ge L^{-1+\epsilon})\le c(\delta,n,\epsilon)L^{-n}.
\end{eqnarray}
\end{proposition}
\begin{proof}[Proof of Proposition \ref{th:Petrov2}] The arguments are
  rather standard, so let us be sketchy. 
Suppose for instance we want to upper bound 
\begin{eqnarray}
  \label{eq:39}
\sup_v  \pi_{D_L}^{\varphi_L}\left(h(v)-\pi_L^{abc}(h(v))\ge L^{-1+\epsilon}\right).
\end{eqnarray}
Define $\varphi'_L=\varphi_L-L^{-1+\epsilon}/2$. 
From \eqref{eq:7}  we see that, under the measure $\pi_L^{abc}$, except
with probability $O(L^{-n})$ one has $h(v)\ge \varphi'_L$ for every $v\in \partial D_L$.
Therefore, 
\begin{multline}
  \label{eq:48}
\sup_v  \pi_{D_L}^{\varphi_L}(h(v)-\pi_L^{abc}(h(v))\ge
  L^{-1+\epsilon})=\sup_v\pi_{D_L}^{\varphi'_L}(
  h(v)-\pi_L^{abc}(h(v))\ge L^{-1+\epsilon}/2)\\
\le \sup_{v}\pi_L^{abc}(h(v)-\pi_L^{abc}(h(v))\ge
L^{-1+\epsilon}/2) +O(L^{-n}).
\end{multline}
In the last step 
we used monotonicity (the increasing event $h(v)-\pi_L^{abc}(v)\ge L^{-1+\epsilon}/2$ becomes more likely if we replace $\varphi'_L$ with a higher boundary condition, see Section \ref{sec:monoton}) and the DLR equations.
Then, Eq. \eqref{eq:7} implies directly \eqref{eq:39}  (estimates are uniform in $v$ because we
assumed that all $v\in D_L$ are uniformly bounded away, by at least
$\delta$, from $\partial \Circle_{abc}$.)

\end{proof}

\subsubsection{Dynamics with hexagonal boundary height}

 \label{rem:ellisse}
 Let $U=\Circle_{abc}$ and the boundary condition $\varphi$ be the
 restriction of $\bp_{abc}$ to $\partial U$. The macroscopic shape
 $\bp$ is \emph{not} non-extremal in $U$ since, while the gradient
 $\nabla \bp$ is well-defined and belongs to $\stackrel\circ\tri$
 everywhere in $\stackrel \circ U$, it approaches $\partial\tri$ when
 the boundary of $U$ is approached.  However, Corollary
 \ref{cor:quasifrozen} is applicable in this case, implying the
 estimate \eqref{eq:provaw} on the time when the equilibrium shape is
 reached.  Just take some sequence $u^{(n)}>0$ tending to zero,
 define \[a^{(n)}=a(1+u^{(n)}),b^{(n)}=b(1+u^{(n)}),c^{(n)}=c(1+u^{(n)})\]
 and let $\varphi^{(n)}$ be the restriction to $\partial
 \Circle_{abc}$ of the macroscopic shape $\bp_{a^{(n)}b^{(n)}c^{(n)}}$
 corresponding to the expanded hexagon
 $\varhexagon_{a^{(n)}b^{(n)}c^{(n)}}$. Since  $\Circle_{a^{(n)}b^{(n)}c^{(n)}}$
contains $U=\Circle_{abc}$ strictly, the macroscopic shape $\bp^{(n)}=\bp_{a^{(n)}b^{(n)}c^{(n)}}$
 is non-maximal in $U$.

The sharp control of the equilibrium measure provided by Theorem
\ref{th:Petrov}, together with the methods developed in the proof 
of Theorem \ref{th:pratica2}, allow to prove the stronger result
$\tmix\approx L^{2+o(1)}$ (Conjecture \ref{th:idealmente2}) in the case where 
$U$ is a closed, simply connected  subset
of the open ellipse $\Circle_{abc}$ and the boundary height $\varphi$ is the restriction of 
$\bar\phi_{abc}$ to $\partial U$. 
\begin{theorem}
\label{th:praticamente} \cite{LTprep} Fix $a,b,c$ with $a+b+c=1$ and
$a>0,b>0,c>0$. Let $U$ be a closed domain contained in $\Circle_{abc}$. Let $\varphi$ be the restriction of 
$\bar\phi_{abc}$ to $\partial U$ and let $(\varphi_L, U_L)_{L\ge1}$ a discretization of
$(\varphi,U)$. Then, $\tmix\le c(\delta)L^{2+\delta}$ for every
$\delta>0$.
\end{theorem}

\section{Local structures of macroscopic shapes}

\label{sec:localstructures}
In this section we formalize and prove Theorem \ref{th:inf2}.

Let  $\bar\phi$ be the macroscopic shape in some domain $U$ with some
boundary height $\varphi$. Consider a point $(x,y)\in \inte U$ where
$ \bar \phi$ is at least twice differentiable, and such that 
$\nabla\bar\phi\in\inte\tri$;  let $H^{\bar\phi}$ be the $2\times 2$
Hessian matrix of
$\bar\phi$ at $(x,y)$. We call 
$\{\nabla\bar\phi,H^{\bar\phi}\}$, the \emph{local structure
  of $\bar\phi$ at $(x,y)$}.
We are excluding points $(x,y)$ where the gradient of
$\bar\phi$ is in $\partial\tri$, or where $\bar\phi$ is non-smooth: in any case, our Theorem \ref{th:pratica2} involves only domains where the macroscopic shape
is non-extremal and in particular is $C^\infty$.

Recall that, if  $\nabla\bar\phi\in\inte\tri$, the components of the
Hessian $H^{\bar\phi}$ verify Eq. \eqref{eq:PDE}, i.e. for any local
structure $\nabla\bp$ and $H^{\bp}$ are related by 
\begin{eqnarray}
  \label{eq:50}
  \sum_{i,j=1}^2a_{ij}(\nabla\bp)H^{\bp}_{ij}=0.
\end{eqnarray}
Therefore, to identify a local
structure it is sufficient to know the gradient of $\bar\phi$ and two
elements of the Hessian matrix, say the $\partial^2_x$ and
$\partial^2_{xy}$ components (i.e. the $(11) $ and $(12)=(21)$ matrix elements). In view of this, we define
\begin{eqnarray}
  \label{eq:28}
  \mathcal A=\{z=(z_1,z_2,z_{11},z_{12})\in \mathbb R^4: (z_1,z_2)\in
  \stackrel\circ\tri
\}
\end{eqnarray}
which should be seen as \emph{the set of all a priori admissible local structures}. Note that
it is not guaranteed that every $z\in \mathcal A$ can be actually
realized as the local structure for some boundary condition.

Let us also define the open set
\[
W=\{w=(a,b,x,y)\in\bbR^4:a>0,b>0,a+b<1,(x,y)\in \Circle_{abc}\}
\]
where as usual it is understood that  $c=c(a,b)=1-a-b$. 
This is the set parametrizing points in  ellipses of the type $\Circle_{abc}$.
\begin{remark}
The allowed values of 
$(a,b)$
belong to the interior of  triangle $\mathbb V=-\tri$, with $\tri$ as in Definition \ref{def:tri}. 
\end{remark}
We introduce a map $f:W\mapsto \mathcal A$ as follows:
\[
f(a,b,x,y)=z=(z_1,z_2,z_{11},z_{12})\in\bbR^4
\]
with $(z_1,z_2)$ the slope $\nabla \bar\phi_{abc}$ at $(x,y)$ (with
$\bar\phi_{abc}$ the macroscopic shape corresponding to the hexagon
$\varhexagon_{abc}$, as in Theorem \ref{th:CLP}) and 
\[
(z_{11},z_{12})=(\partial^2_x \bar\phi_{abc} , 
\partial^2_{xy}\bar\phi_{abc})\in\bbR^2\] with the derivatives
computed at $(x,y)$.  Note that $f(W)$ is the set of all local
structures arising from 
macroscopic shapes with  boundary heights of ``hexagonal type''. A priori it could be  that $f(W)$ is a proper
subset of $\mathcal A$, and even that  $f(W)$ has
topological dimension smaller than $4$. Indeed, hexagonal boundary
conditions look very special in the class of all admissible boundary heights.
However, Theorems \ref{jacob2} and
\ref{th:crazy1} below exclude
these possibilities.

Let the $4\times 4$
matrix
\begin{eqnarray}
Df(w)=\left(  \begin{array}{ccccc}
   \partial_a z_1 &  \partial_b z_1 &  \partial_x z_1 &  \partial_y z_1 \\
 \partial_a z_2 & \partial_b z_2 &  \partial_x z_2 & \partial_y z_2 \\
\partial_a z_{11}&\partial_b z_{11}& \partial_x z_{11} & \partial_y z_{11} \\
\partial_a z_{12}&\partial_b z_{12} &\partial_x z_{12} & \partial_y z_{12} 
  \end{array}
\right)
\label{eq:Df}
\end{eqnarray}
denote the derivative of $f$ at $w\in W$. If, for some
$w\in W$, $Df(w)$ has (maximal) rank $4$,
then $f$ is locally bijective on $\mathcal A$: every point $z$ in a suitable  neighborhood
$B(f(w),\epsilon)\cap \mathcal A$  (with $B(z_0,r)$ the ball of radius
$r$ centered at $z_0$) has a unique pre-image
through $f$ in $W$ at distance $O(\epsilon)$
from $w$. 
 We have
\begin{theorem}
  \label{jacob2} The rank of $Df(w)$ is $4$ for every $w\in W$
. More
  precisely, in compact subsets of
  $W$ the 
  determinant of $Df(w)$ is strictly negative.
\end{theorem}

\begin{proof}[Proof of Theorem \ref{jacob2}]
  Call simply $\bar\phi$ the macroscopic shape in the $abc$ hexagon
  and recall that, here and in the following, $c=1-a-b$.  The explicit
  expression for $\nabla \bar \phi$ is given in
  \cite{CLP}\footnote{The authors of \cite{CLP} consider on $P_{111}$
    an orthogonal coordinate frame $(u,v)$ that does not coincide with the
    non-orthogonal coordinate frame $(x,y)$ we adopt here. The change of coordinates is given in
\eqref{eq:29}.}.  Introduce the coordinates
  $(u,v)$ and $(u',v')$ as
\begin{gather}
  \label{eq:29}
  u=u(x,y)=-x+\frac y2,\quad
v=v(x,y)=-\frac{\sqrt 3}2y\\
u'=u'(x,y)=y - \frac{x}{2}=u/2-\frac{\sqrt3}2v,\quad
v'=v'(x,y)=-\frac{\sqrt{3}}{2}x=\frac{\sqrt 3}2u+\frac v2.
\end{gather}
What one finds is then
\begin{gather}
  \label{eq:s,t}
z_1=  \partial_x \bp=-\frac1\pi\cot^{-1}\left[
\frac{Q_{bca}(u(x,y),v(x,y))}{\sqrt{E_{bca}(u(x,y),v(x,y))}}
\right]
\\
z_2=\partial_y \bp=-\frac1\pi\cot^{-1}\left[
\frac{Q_{cab}(u'(x,y),v'(x,y))}{\sqrt{E_{cab}(u'(x,y),v'(x,y))}}
\right]\\
Q_{abc}(u,v)=\frac{\sqrt3}2 \left(\frac43 v^2 - 4 u^2 + b^2 + a b + b c - 
   a c\right),
\\
E_{abc}(u,v)=3a b c -(3(a+c)^2u^2-2\sqrt3(a+2b+c)(a-c)u
v\\\nonumber +((a+2b+c)^2-4a c)v^2).
\end{gather}
For later convenience, let us point out also that
\begin{eqnarray}
  \label{eq:32}
 z_3:= -1-z_1-z_2=-\frac1\pi\cot^{-1}\left[
\frac{Q_{abc}(u''(x,y),v''(x,y))}{\sqrt{E_{abc}(u''(x,y),v''(x,y))}}
\right]
\end{eqnarray}
with
\begin{eqnarray}
  \label{eq:33}
  u''(x,y)=\frac{1}{2}x+\frac{1}{2}y=-\frac{u}2-\frac{\sqrt3}2v,\quad
  v''(x,y)=-\frac{\sqrt{3}}{2}x + \frac{\sqrt{3}}{2}y=\frac{\sqrt3}2 u-\frac{v}2.
\end{eqnarray}

Let also $z_{11}=\partial_x z_1, z_{12}=\partial_y z_1=\partial_x z_2$.
\begin{remark}
\label{rem:tuttoliscio}
The boundary of the ellipse $\Circle_{abc}$ corresponds to the set of zeros of
$E_{bca}(u(x,y),v(x,y))$. One can also check that
\[E_{bca}(u(x,y),v(x,y))=E_{cab}(u'(x,y),v'(x,y))=
E_{abc}(u''(x,y),u''(x,y))\] and formulas above are understood to hold inside $\Circle_{abc}$. From the explicit formulas \eqref{eq:Df} and \eqref{eq:s,t} one sees
that both $f(w)$ and $Df(w)$ are $C^\infty$ in $W$, with uniform bounds when $w$ is in  compact subsets of $W$.  
\end{remark}

We have to prove that the determinant of the matrix 
$Df$
is negative. Observe that, when we take derivatives with respect to
$a$ or $b$, we have to remember that $c$ is a function of $a,b$. One
can painfully check\footnote{It is immediate from Eqs. \eqref{eq:s,t}
  that the matrix elements of $Df$, and therefore also  the determinant, are
  rational functions of $a,b,x,y$. For the actual computation of the
  coefficients of the two polynomials we used Mathematica, in order to
symbolically  simplify otherwise intractable expressions.} that the determinant of $Df$
equals
\begin{gather}
  \det(Df)=\frac1{32\pi^4}\frac{N_{}(x,y)}{D_{}(x,y)}\\
N_{}(x,y)=
(1 - a) (1 - b) (1-c) + 2 (1 - a^2) y^2 + 2(1-c^2) x^2 - 
  4 (1 - a) (a+b) x y
\\
D_{}(x,y)=
\left(y^2 - \frac{(a + b)^2}4\right) \left(x^2 - \frac{(1 - a)^2}4\right)^2\\
\nonumber \times \left(x - (y - \frac{1 - b}2)\right)^2 \left(x - (y + \frac{1 - b}2)\right)^2.
\end{gather}
One easily sees that $D_{}(x,y)$ vanishes exactly along the 
sides of the hexagon $\varhexagon_{abc}$, and is negative inside the hexagon (since in the hexagon the $y$ coordinate ranges between $-(a+b)/2$ and $+(a+b)/2$).
As for the numerator, it vanishes for  
\begin{eqnarray}
  y=-\frac{2(1-a)(a+b)x\pm \sqrt2\sqrt{(1-a)(1-b)(a+b)(-1+a^2-4x^2)}}{2(a^2-1)}
\end{eqnarray}
Since $a<1$, the square root is imaginary and
therefore the numerator has no zeros. The numerator is clearly
positive for $(x,y)=(0,0)$, so it is positive everywhere.
\end{proof}

A key point for the following is that $f:W\mapsto \mathcal A$ is actually
a bijection:
\begin{theorem}
\label{th:crazy1}
  The application $f$ is a diffeomorphism from $W$ to $\mathcal A$. In particular, $f(W)=\mathcal A$.
\end{theorem}

The non-trivial step is to prove that points on the boundary of $W$ are
mapped through $f$ to points on the boundary of $\mathcal A$
(Proposition \ref{th:bordobordo}). Given this, the proof of Theorem
\ref{th:crazy1} follows rather closely that of a theorem of Hadamard \cite{Hadamard},
 that gives a necessary and sufficient condition for a smooth
map from $\bbR^n$ to $\bbR^n$ to be a diffeomorphism, cf. for instance
 \cite{Gordon}.

\begin{proposition}[Compact sets have compact pre-images]
\label{th:bordobordo}
  Let $\{w_n\}_{ n\ge1}$ be a sequence of points in $W$, that tends  as $n\to\infty$ to a
  point $\bar w$ on  the
  boundary of $W$. Then, none of the sub-sequential limits of the
  sequence $\{f(w_n)\}_{ n\ge1}$
 is in  $\mathcal A$. 
\end{proposition}
\begin{proof}[Proof of Proposition \ref{th:bordobordo}]

Recall that, for $w\in W$, we write $w=(a,b,x,y)$ and $f(w)=(z_1,z_2,z_{11},z_{12})$. Note that
\begin{eqnarray}
  \label{eq:34}
  \partial W=\{w: (a,b)\in \partial \mathbb V\}\cup \{w:(a,b)\in \stackrel
  \circ {\mathbb V}, (x,y)\in \partial \Circle_{abc}\}.
\end{eqnarray}
Therefore, if $\bar w=(\bar a, \bar b,\bar x, \bar y)=\lim_n
w_n\in\partial W$, exactly one of these two
conditions holds:
\begin{enumerate}
\item [(A)]$(\bar a,\bar b)\in \stackrel \circ {\mathbb V}$
  and $(\bar x,\bar y)$ is on the boundary of $\Circle_{\bar a\bar b\bar c}$;
\item [(B)] $(\bar a,\bar b)$ is in $\partial \mathbb V$.
\end{enumerate}
We have to prove that, in both cases, at least one of the following two
options occurs: 
\begin{enumerate}
\item [(Option 1)] $(z_1,z_2)$ approaches the boundary of $\tri$  as $n\to\infty$;
\item [(Option 2)] the directional derivative $|\partial_v z_i|$ diverges  as
  $n\to\infty$, for some $i=1,2,3$ and for some direction $v$ in the
  plane. This implies that either $(z_1,z_2)$ approaches $\partial \tri$,
  or $z_{11}^2+z_{12}^2$ 
  diverges, in both cases implying the statement of the Proposition. 

Indeed, recall that $z_3=-1-z_1-z_2$ and observe that $\partial_v z_i$
is a linear combination of $\partial_x z_i $ and $\partial_y z_i$. If
$|\partial_v z_i|$ diverges, then  one among $z_{11}=\partial_x
z_1$, $z_{12}=\partial_x z_2=\partial_y z_1$ or $\partial_y z_2$ diverge. 
If either $z_{11}$ or $z_{12}$ diverges, we are done. So suppose instead that
$\partial_y z_2$ diverges. Remember that
\begin{multline}
a_{11}(z_1,z_2)\partial_x z_1+2a_{12}(z_1,z_2)\partial_x
z_2+a_{22}(z_1,z_2)\partial_y z_2\\
=a_{11}(z_1,z_2)z_{11}+2a_{12}(z_1,z_2)z_{12}+a_{22}(z_1,z_2)\partial_y z_2
=0
\end{multline}
and that, when $(z_1,z_2)$ is bounded away from
$\partial \tri$, 
$a_{ij}$ are finite and $a_{ii}$ are strictly positive. As a
consequence, 
if $\partial_y z_2$ diverges then either $(z_1,z_2)$ approaches
$\partial \tri$ or at least one among $z_{11}$
and $z_{12}$ also diverges.
\end{enumerate}

In Case (A), it follows directly from Theorem \ref{th:CLP} that the
slope $(z_1,z_2)$ approaches the boundary of $\tri$, so Option 1
occurs.

In Case (B) 
we have to go back to formulas \eqref{eq:s,t}-\eqref{eq:32} for
$z_1,z_2$ and $z_3=-1-z_1-z_2$, that we rewrite compactly as
\begin{eqnarray}
  \label{eq:35}
  z_i(x,y)=-\frac1\pi\cot^{-1}\left[\frac{Q_{abc}^{(i)}(x,y)}{\sqrt{E_{abc}^{(i)}(x,y)}}\right],\quad i=1,2,3
\end{eqnarray}
(recall that actually $E_{abc}^{(i)}$ does not depend on $i$, so we will
just write $E_{abc}$).
The numerators $Q_{abc}^{(i)}$ are
second-order polynomials in $x,y$, symmetric under
$(x,y)\leftrightarrow (-x,-y)$. Let $Z_{abc}^{(i)}, i=1,2,3$ be the
respective level-zero sets on the $(x,y)$ plane: they are
hyperbolas, that can be degenerate (two
straight lines intersecting at $(0,0)$)
for particular values of $a,b,c=1-a-b$. More precisely, $Z_{abc}^{(i)}$ is
degenerate if and only if $Q_{abc}^{(i)}(0,0)=0$.  It is however easy to check that
there are no values $a,b,c=1-a-b$ for which the three hyperbolas are
simultaneously degenerate: an explicit calculation
shows  that $\sum_{i=1}^3 Q_{abc}^{(i)}(0,0)$ is never zero. Therefore, 
for $n\to\infty$, at least one of the curves $Z_{a_n b_n c_n}^{(i)}$
tends to a non-degenerate
hyperbola. To fix ideas, let us assume that this is the case for
$i=1$, i.e. that 
\begin{eqnarray}
  \label{eq:38}
  \liminf_n |Q_{a_nb_nc_n}^{(1)}(0,0)|>0.
\end{eqnarray}

Then we proceed as follows. We first note that  
  the sup-norm of $E_{a_nb_nc_n}$ on
  $\Circle_{a_n b_n c_n}$ is just $\delta_n:=E_{a_nb_nc_n}(0,0)=3a_n b_n
  c_n=o(1)$ (the graph of $E_{abc}$ is a concave paraboloid with
  gradient zero at $(0,0)$ and vanishes at the boundary of $\Circle_{abc}$).
The fact that $a_nb_nc_n=o(1)$ is because when $(a_n,b_n)$ approaches
$\partial\mathbb V$, at least one of the three values
$a_n,b_n,1-a_n-b_n$ approaches zero.
Given $(x_n,y_n)\in \Circle_{a_n b_n c_n}$, if  for some $i=1,2,3$
\[\left|\frac{Q_{a_nb_nc_n}^{(i)}(x_n,y_n)}{\sqrt{E_{a_nb_nc_n}(x_n,y_n)}}\right|\ge\frac1{\delta_n^{1/8}}
\]
then, for $n\to\infty$, $z_i$
approaches either $0$ or $1$ (according to the sign of $Q_{a_nb_nc_n}^{(i)}/\sqrt{E_{a_nb_nc_n}} $), and
therefore Option 1 occurs.

Assume instead that at $(x_n,y_n)$ one has
  \begin{eqnarray}
    \label{eq:30}
\max_i |Q_{a_nb_nc_n}^{(i)}/\sqrt{E_{a_nb_nc_n}}|\le\frac1{\delta_n^{1/8}}.
  \end{eqnarray}
Note that automatically $(x_n,y_n)$ is bounded away from $(0,0)$, otherwise condition
\eqref{eq:30} would be violated, since  one would
have \[
\left|\frac{
Q_{a_nb_nc_n}^{(1)}(x_n,y_n)}{\sqrt{E_{a_nb_nc_n}(x_n,y_n)}}\right|\ge
\frac{\left|
Q_{a_nb_nc_n}^{(1)}(x_n,y_n)\right|}{\sqrt{\delta_n}}
\simeq
\frac{\left|Q_{a_nb_nc_n}^{(i)}(0,0)\right|}{\sqrt{\delta_n}}\]
and $Q_{a_nb_nc_n}^{(1)}(0,0)$ is bounded away from zero for $n\to\infty$, cf. \eqref{eq:38}.
We look at the derivative of $z_i$ at $(x_n,y_n)$ in the direction $v$ tangent to  the local level
line of $E_{a_nb_nc_n}$:
we get \[\partial_v
  z_i=-\frac1\pi\frac1{1+(Q_{a_nb_nc_n}^{(i)}/\sqrt{E_{a_nb_nc_n}})^2}\frac{\partial_v
    Q_{a_nb_nc_n}^{(i)}}{\sqrt{E_{a_nb_nc_n}}}.\]
Therefore, 
\begin{eqnarray}
  \label{eq:36}
  |\partial_v z_i|\ge const.\times \delta_n^{2/8-1/2}|\partial_v
    Q_{a_nb_nc_n}^{(i)}|=const.\times \delta_n^{-1/4}|\partial_v
    Q_{a_nb_nc_n}^{(i)}|.
\end{eqnarray}
If we can prove that $|\partial_v
    Q_{a_nb_nc_n}^{(i)}|$ stays bounded away from zero as $n\to\infty$ for at
    least one value of $i$, we get that $|\partial_v z_i|$ diverges
    and we can conclude that Option 2 occurs.
    To control $|\partial_v
    Q_{a_nb_nc_n}^{(i)}|$, observe that 
    \begin{equation}
      \label{eq:37}
      \begin{split}
      \nabla Q_{a_nb_nc_n}^{(1)}&=2\sqrt3(-2x+y,x)\\
      \nabla Q_{a_nb_nc_n}^{(2)}&=2\sqrt3(y,-2y+x)\\
      \nabla Q_{a_nb_nc_n}^{(3)}&=2\sqrt3(-y,-x).
   \end{split}  \end{equation}
From these explicit formulas it is immediate to check that, whenever $(x,y)\ne (0,0)$, all three
gradients have non-zero norm and  that there are at
least two of them  that are not colinear\footnote{Just compute
  $\nabla Q_{a_nb_nc_n}^{(i)}\cdot [\nabla Q_{a_nb_nc_n}^{(j)}]^\bot$ for all $i\ne j$, with $v^\bot$ the vector $v$ rotated by $\pi/2$, and
check that only for $(x,y)=0$ the three products vanish simultaneously.
}. As a consequence (recalling that $(x_n,y_n)$ is bounded away from
$(0,0)$), for any
given direction $v$ one has that $|\partial_v
    Q_{a_nb_nc_n}^{(i)}|$ is bounded away from $0$ for at least one
    value of $i$, as we wished to show.

\end{proof}

\begin{proof}[Proof of Theorem \ref{th:crazy1}] 

{\bf Point (I): $f$ is surjective ($f(W)=\mathcal A$).}
Fix some $\bar w\in W$ and let $\bar z=f(\bar w)$. We define on $\mathcal A$
the radial vector field $v(z)=\bar z-z$ that points everywhere towards $\bar z$. Given $z\in \mathcal A$, we let for
$t\ge 0$
\begin{eqnarray}
  \label{eq:yz}
y^z(t)=\bar z+e^{-t}(z-\bar z),  
\end{eqnarray}
that solves
\[
\frac d{dt}y^z(t)=v(y^z(t)),\quad y^z(0)=z
\]
and note that $\lim_{t\to\infty} y^z(t)=\bar z$.
Thanks to Theorem \ref{jacob2},  $\bar z$ is in the interior of  $f(W)$, so there
exists $0\le\tau^z<\infty$ such that  $y^z(\tau^z) \in f(W)$. Let
$w_{\tau^z}\in W$ be such that  $f(w_{\tau^z}) = y^z(\tau^z)$. We let $w^z(t)$ be the
solution of the differential equation on $W$
\begin{eqnarray}
  \label{eq:flusso}
\left\{
  \begin{array}{l}
    \frac d{dt }w^z(t) =V(w^z(t))\cdot
v(f(w^z(t)))\\
w^z(\tau^z)=w_{\tau^z},
  \end{array}
\right.
\end{eqnarray}
where
\begin{eqnarray}
  V(w)= [D f(w)]^{-1}.
\end{eqnarray}
The existence of the inverse of the matrix $D f(w)$ is guaranteed by
Theorem \ref{jacob2}.
The solution $w^z(t)$ exists at least locally around
$t=\tau^z$. 

Let $I$ be the interval of definition of the solution, and $\underline t=\inf\{s:s\in I\}$.
For every $t \in \bbR_+ \cap I $ we have $f(w^z(t)) = y^z(t)$ since, as
one easily checks,  both quantities verify the same differential equation
and take the same value for $t=\tau^z$. We wish to show that
$\underline t<0$,
so that $f(w^z(0)) = z$, i.e. $z\in f(W)$ and in turn (by the arbitrariness of
$z$) this implies $f(W)=\mathcal A$.

Let us assume by contradiction that $\underline t\ge0$. Recalling Remark
\ref{rem:tuttoliscio} and  Theorem
\ref{jacob2}, we have that  $Df(w)$ is
$C^\infty$ (actually analytic) in $W$ and $\det(D f(w))$ is bounded away from zero in
compact subsets of $W$. Therefore, the vector field $V(w)\cdot v(f(w))=[D
f(w)]^{-1}\cdot v(f(w))$ is $C^\infty$ and bounded, away from the boundary of $W$.
As a consequence, we have that there exists a sequence $s_n \searrow
\underline t$ such that $w^z(s_n)$ approaches $\partial W$ as
$n\to\infty$ (otherwise the solution could be extended to short times before $\underline t$).  By Proposition \ref{th:bordobordo}, one deduces that
the sequence $f(w^z(s_n))=y^z(s_n)$ cannot have a limit in $\mathcal
A$. However, from \eqref{eq:yz} such limit exists and is simply $\bar z+e^{-\underline
  t}(z-\bar z)$, which belongs to $\mathcal A$ (recall that $z,\bar z$
are in $\mathcal A $, remark that $\mathcal A$ is convex and that $\bar z+e^{-\underline
  t}(z-\bar z)$ is a
convex combination of $z$ and $\bar z$).

{\bf Point (II): $f$ is bijective and  a diffeomorphism.}
We know from point (I) that $f$ is surjective, and from Theorem \ref{jacob2} that
it is a local diffeomorphism.
It remains only to prove that $f^{-1}(\bar z)$ is uniquely defined for
every $\bar z\in \mathcal A$ (injectivity).  This is essentially
identical to the proof of injectivity in Hadamard's theorem
(cf. \cite[Theorem A]{Gordon}), so we will just sketch the main steps.

First, the set $f^{-1}(\bar z)=\{w\in W:f(w)=\bar z\}$ is finite:
otherwise, by Proposition \ref{th:bordobordo} (compact sets have
compact pre-images) it would contain an accumulation point $w_\infty$
in $W$. This would contradict Theorem \ref{jacob2}, since the determinant of
$Df(w_\infty)$ is non-zero, so that $f$ is locally one-to-one in a
neighborhood of $w_\infty$.  

Second, to each $w_i\in f^{-1}(\bar z)$ is associated the
set \[W_i=\{w_0\in W: \lim_{t\to\infty}w(t)=w_i\},\] with $w(t)$ the
solution of the Cauchy problem 
\begin{eqnarray}
  \label{eq:flusso2}
\left\{
  \begin{array}{l}
    \frac d{dt }w(t) =V(w(t))\cdot
v(f(w(t)))\\
w(0)=w_0.
  \end{array}
\right.
\end{eqnarray}  Recall that $f(w(t))= y^{z_0}(t)$, with $y^{z_0}(t)$ defined in
\eqref{eq:yz} and $z_0=f(w_0)$. Since $y^{z_0}(t)$ stays in a compact set
uniformly for $t\ge0$, 
using again Proposition
\ref{th:bordobordo} we see that $w(t)$ exists for all positive times
(it never approaches the boundary of $W$). But 
$y^{z_0}(t)$ converges to $\bar
z$ as $t\to\infty$, so that $w(t)$ tends to an inverse of $\bar z$: thanks to
the arbitrariness of $w_0$, this implies
that $W=\cup_i W_i$. Moreover, each
$W_i$ is open,  by continuity of solutions of \eqref{eq:flusso2}
with respect to initial conditions. Given that the $W_i$ are disjoint
and that $W$ is open and connected, one deduces that $f^{-1}(\bar z)$
contains a single element.
\end{proof}

\section{Monotonicity and constrained dynamics}
\label{sec:monoton}
As well as in previous works on lozenge dynamics
\cite{Wilson,CMT,CMST,dominos}, monotonicity will play an important
role. Let us briefly recall the basic idea.

 In the set of stepped monotone interfaces we introduce a partial
order where $h\le h'$ if $h(v)\le h'(v)$ for every $v$. It is
well known that dynamics conserves the partial order: give a discrete domain $U_L$, it is possible
to couple in the same probability space all the evolutions
$h^{\eta;\varphi_L}_t$ with boundary height $\varphi_L$ and   initial condition $\eta\in \Omega_{U_L,\varphi_L}$  in such a way that, $\bbP$-almost surely,
\begin{eqnarray}
  \label{eq:19}
  h^{\eta,\varphi_L}_t\le h^{\eta',\varphi'_L}_t \quad\text{for every
  }\quad t\ge 0, \quad \text{if}\; \eta\le \eta'  \; \text{and}\; \varphi_L\le \varphi'_L.
\end{eqnarray}

An immediate consequence on the equilibrium measures is that $\pi^{\varphi_L}_{U_L}$ is
stochastically dominated by $\pi^{\varphi'_L}_{U_L}$.

\medskip

Consider two stepped monotone surfaces with height functions $h^-,h^+$
such that $h^-\le h^+$. Let moreover $U_L$ be a discrete domain and
$\varphi_L$ be a boundary height such that $h^-\le \varphi_L\le h^+$
on $U^{ext}$.
The dynamics in $U_L$ with boundary height $\varphi_L$, ``ceiling''
$h^+$ and ``floor'' $h^-$ is defined as the usual dynamics $h_t$,
except
that any updates that would lead to a violation of the inequalities
\[
h^-\le h_t\le h^+
\]
are discarded (censored).
Of course, we will assume that the initial condition $\eta$ does satisfy
$h^-\le \eta \le h^+$. The invariant measure of the constrained
dynamics is simply the uniform measure $\pi_{U_L}^{\varphi_L}$
conditioned on the interface being between floor and ceiling, i.e. 
\[
\pi_{U_L}^{\varphi_L}(\cdot|h^-\le \cdot\le h^+).
\]

Define the distance between floor and ceiling as \[\max_{v\in
  U_L}(h^+(v)-h^-(v)).\] Then:
\begin{lemma}
\label{lemma:Wilson}\cite[Theorem 4.3]{CMT}
The Glauber dynamics in a discrete domain $U_L$ of diameter $D$ in the graph-distance, with floor and ceiling at distance $H/L$, has 
 $\tmix=O(D^2 H^2(\log D)^2)$.  
\end{lemma}

Take $U_L$ to be the discretization of a domain $U$, so that its graph-distance diameter $D$ is of order $L$.
Note that, if we let $h^\pm=\varphi_L\pm A $ with $A=O(1)$ sufficiently
large then the constrained
dynamics exactly coincides with the unconstrained one. This is simply
because the height functions $h$ and $\varphi_L$ change by $\pm 1/L$
or $0$ along edges of $\mathcal T_L$: if
$u\in \partial U_L$ and $v\in U^{int}$ and $u,v$ are at graph-distance
$d(u,v)$, one has
\[h(v)\le h(u)+d(u,v)/L\] and \[h^+(v)\ge h^+(u)-d(u,v)/L=\varphi_L(u)+A
-d(u,v)/L.\] Given that $h(u)=\varphi_L(u)$ on $\partial U_L$, we see
that
\[
h^+(v)\ge h(u)+A-d(u,v)/L\ge h(v)+A-2d(u,v)/L.\]
If $A$ is chosen larger than
$2D/L=O(1)$ we have then $h(v)\le
h^+(v)$ (and analogously $h(v)\ge h^-(v)$)  deterministically. Hence 
the floor/ceiling constraints are automatically satisfied by the unconstrained dynamics.
On the other hand, if $h^\pm=\varphi_L\pm A$ then  the distance between floor and ceiling is $2A$.
Therefore, an immediate consequence of Lemma
\ref{lemma:Wilson} is:
\begin{corollary}
\label{cor:tmix}
Let the discrete domain $U_L$  be a discretization of a domain $U$.
For any boundary height $\varphi_L$,
  the mixing time of the Glauber dynamics with neither floor nor ceiling
  is smaller than $C\,L^4 (\log L)^2$ for some constant $C$ depending only on $U$.
\end{corollary}

\section{Proof of Theorem \ref{th:pratica2}
}
\label{sec:provadin}


Here we make a few comments about the idea of the proof and its structure.
Recall from the introduction that we want to show that the interface stays with very high probability
``trapped'' between two deterministic surfaces that evolve on  a time
scale just slower than diffusive
and both tend to the macroscopic shape.
We will only consider the upper bound in the following because the proof of the lower bound is identical.

The first step, that does not require much work, is to realize that it
is enough to prove that when the initial condition is at distance
$2\epsilon_L$ from equilibrium (for some suitably small $\epsilon_L$) then within time $L^{2+o(1)}$
the interface reaches distance $\epsilon_L$ (see Claim \ref{claim:i}). 
To prove this, the key point is Claim \ref{claim:j}, that says that the height function stays with high
probability below the
deterministically evolving interface  \[\tilde\phi_t:=\bp + \epsilon_L
(1-t/L^{2+o(1)}) \psi,\] with a well chosen function $\psi\ge2$, until the
time when $(1-t/L^{2+o(1)})  $ becomes sufficiently small. As mentioned
in the introduction, we prove the bound by looking at ``mesoscopic''
regions of size $L^{-1/2+o(1)}$ and at time increments $L^{1+o(1)}$,
that are
small with respect to the
diffusive scaling. The choice of $\psi$
will be justified in  Remark \ref{rq:choix_psi}. In practice, one must
guarantee that $\hat{\mathcal L} \psi$ (with $\hat{\mathcal L}$ the
linearization of the elliptic operator
in \eqref{eq:PDE}, that should determine the interface drift, see \eqref{eq:5}) is comparable with $
\Delta\psi$ (with $\Delta$ the usual Laplacian).

\medskip

For simplicity we will write $\pi_L$ for the equilibrium measure
$\pi_{U_L}^{\varphi_L}$ and as usual $\bar\phi$ denotes the
macroscopic shape in $U$ with boundary height $\varphi$.

Let $\epsilon_L=1/\log L$.
To prove Theorem \ref{th:pratica2} it is sufficient
to prove that
\begin{eqnarray}
  \label{eq:massima}
  \sup_{t>c(\delta)L^{2+\delta}}\bbP(\exists v\in U_L: h_t(v)-\bp(v)>2\epsilon_L)=o(1)
\end{eqnarray}
and 
\begin{eqnarray}
  \label{eq:minima}
  \sup_{t>c(\delta)L^{2+\delta}}\bbP(\exists v\in U_L: h_t(v)-\bp(v)<-2\epsilon_L)=o(1),
\end{eqnarray}
with bounds uniform in the initial condition $\eta$ (we  omit for lightness
the argument $\eta$ in $h^\eta_t$).
We will prove only \eqref{eq:massima}, the proof of \eqref{eq:minima} being essentially identical.

Let $G$ be a positive constant, independent of $\eta$ and $L$,  that will be fixed in a moment (it will depend only on the diameter of $U$).
We have:
 \begin{claim}
\label{claim:i}
   For 
 $i=0,1,\dots,\lfloor G/\epsilon_L\rfloor-1$,
 \begin{eqnarray}
   \label{eq:9}
\bbP( h_t\le G-i\epsilon_L +\bar\phi \text{ for every } t\in  [T_i, L^5]
)\ge 1-i/L
 \end{eqnarray}
where \[
T_i=    L^{2+\delta/2} i.\]
 \end{claim}
   When we write $h_t\le g$ like in
   \eqref{eq:9},  what we mean exactly is that for every $u\in U_L\cap \mathcal T_L$ one has
   $h_t(u)\le g(u)$.

   \begin{proof}[Proof of \eqref{eq:massima}  given Claim 
\ref{claim:i}]
   For $i=\lfloor G/\epsilon_L\rfloor-1$ we
   have \[T_i=O(L^{2+\delta/2}/\epsilon_L)=O(L^{2+\delta/2}\log L)\ll c(\delta)L^{2+\delta} \] and we obtain 
   that
   \begin{eqnarray}
     \label{eq:ded}
\bbP(h_t\le \bp+2\epsilon_L\quad \text{for every } t\in
   [c(\delta)L^{2+\delta},L^5])=1+o(1).     
   \end{eqnarray}
   On the other hand, we know from Corollary \ref{cor:tmix} that the
   mixing time of the dynamics is $\tmix=O(L^4(\log L)^2)$. Therefore,
   from \eqref{eq:18} we see that for times larger than $L^5$ the
   system is at equilibrium (modulo a negligible error term
   $O(\exp(-L/(\log L)^2))$, uniform in time and in $\eta$) and we deduce that
   \begin{eqnarray}
     \pi_L(h\le \bp+2\epsilon_L)=1+o(1)
   \end{eqnarray}
so that 
\begin{eqnarray}
\label{eq:ded2}
  \sup_{t>L^5}\bbP(h_t\le \bp+2\epsilon_L)=\pi_L(h\le \bp+2\epsilon_L)+O(\exp(-L/(\log L)^2))=1+o(1).
\end{eqnarray}
Equations \eqref{eq:ded} and \eqref{eq:ded2} imply \eqref{eq:massima}.
It will be clear from the proof of Claim \ref{claim:i} that in \eqref{eq:9} we could have replaced $L^5$ with any other
larger power of $L$.
   \end{proof}

   \begin{proof}[Proof of Claim \ref{claim:i}]

We prove \eqref{eq:9} by induction on $i$. 
The functions  $h_t $ and $\bp$ are uniformly $1$-Lipschitz in space and they 
coincide on the boundary of $U_L$: therefore,  Eq. \eqref{eq:9}
for $i=0$ is trivially true (for every $t\ge 0$) if $G$ is chosen large enough depending
 on the diameter of $U$.

 \begin{definition}
\label{def:psi}
Set for $(x,y)\in\bbR^2$
\begin{eqnarray}
\label{eq:psi}
  \psi(x,y)=\psi(0,0)-e^{x/\xi}-e^{y/\xi}
\end{eqnarray}
with $\xi\ll 1$ but independent of $L$ and the constant $\psi(0,0)$ chosen so that $\inf\{\psi(x,y),(x,y)\in
U\}=2$. 
Let $\psi_{max}=\max\{\psi(x,y):(x,y)\in U\}$ and $N\sim L(1-1/\psi_{max})$ the smallest
integer such that 
\[
\left(1-\frac NL\right)\psi_{max}\le 1.
\]   
 \end{definition}

To prove \eqref{eq:9} for $i+1$ given the same statement for  $i$, we  proceed as follows. 
Since $\psi\ge2 $ in $U$, the inductive hypothesis (i.e. \eqref{eq:9} for $i$) implies
\begin{eqnarray}
  \label{eq:11}
\bbP(  h_t\le {G-(i+2)\epsilon_L}+\bar\phi+\epsilon_L \,\psi, \; \text{
  for every } T_i\le t\le L^5)\ge 1-i/L.
\end{eqnarray}

\smallskip

 Define 
\begin{eqnarray}
  \label{eq:13}
  \gamma_{i,j}={G-(i+2)\epsilon_L}+\bar\phi+ \epsilon_L\,
  (1-j/L)\psi,
\end{eqnarray}
let $T_{i,j}=T_i+
j L^{1+\delta/4}$
and $E_{i,j}$ be the event
\begin{eqnarray}
  \label{eq:41}
  E_{i,j}=\{h_t\le \gamma_{i,j}\;\text{ for every }  T_{i,j} \le
  t\le L^5\}.
\end{eqnarray}
We will prove:
\begin{claim}
  \label{claim:j}
Fix $0\le i <\lfloor G/\epsilon_L\rfloor-1$ and assume that \eqref{eq:9} holds. 
For $j\le N$,
\begin{eqnarray}
  \label{eq:12}
\bbP(  E_{i,j})\ge 1-i/L-j/L^3.
\end{eqnarray}
\end{claim}
Taking $j=N$, we obtain the claim \eqref{eq:9} for $i+1$, since
$T_{i,N}\le T_{i+1}$, \[\gamma_{i,N}\le G-(i+1)\epsilon_L+\bp\]  and \[1-i/L-N/L^3\ge 1-(i+1)/L.\] This concludes the proof of Claim
\ref{claim:i}, assuming Claim \ref{claim:j}.

   \end{proof}

   \begin{remark}\label{rq:choix_psi}
     The choice of $\psi$, which might look at first sight rather
     arbitrary, is dictated by the following reasoning. At time $T_i$
     we have $h_t\le \bp+\epsilon_L\psi+c$, with $c$ the constant
     $G-(i+2)\epsilon_L$.  From the discussion in the Introduction, we
     expect the macroscopic evolution of the interface under diffusive
     time scaling to be given by \eqref{eq:5}.  Linearizing the
     differential operator $\mathcal L$ around $\bp+c$ and observing
     that $\mathcal L(\bp+c)=0$, we  find that
\[
\mathcal L (\bp +c+\epsilon_L \psi)=\epsilon_L \hat{\mathcal L}\psi+O(\epsilon_L^2)
\]
with $\hat {\mathcal L}$ the linear elliptic operator
\begin{multline}
  \label{eq:linearizzato}
\hat{\mathcal L}\psi=\sum_{i,j=1}^2 a_{ij}(\nabla\bp)\partial^2_{ij}\psi\\+
\sum_{i,j=1}^2 \left[(\partial_x \psi)\,\partial_sa_{ij}(s,t)|_{(s,t)=\nabla\bp}+(\partial_y \psi)\,\partial_ta_{ij}(s,t)|_{(s,t)=\nabla\bp}
\right].
  \end{multline}
Now observe that the Hessian matrix of our $\psi$ 
is diagonal, with negative
diagonal entries:
\begin{eqnarray}
  \label{eq:8}
  H^\psi=-\frac1{\xi^2}\left(
    \begin{array}{cc}
   e^{x/\xi} & 0\\
0 &  e^{y/\xi}
    \end{array}
\right).
\end{eqnarray}
Therefore, the first sum in \eqref{eq:linearizzato} is strictly and pointwise  negative, uniformly in $U$ (the diagonal elements $a_{ii}$ are positive) and the 
second sum can be neglected, if $\xi$ is small (because $|\partial_x \psi|\ll
|\partial^2_x \psi|$ and similarly for the $y$ derivatives).
In conclusion, with our choice of $\psi$, $\mathcal L(\bp+c+\epsilon_L \psi)$ is everywhere negative, so the interface feels a negative drift that pushes it towards
the equilibrium shape. The drift is of order $-L^{-2}$ 
 (recall that, in \eqref{eq:5}, $\tau$ is  the rescaled time $\tau=t/L^2$).
This heuristic reasoning is what is behind  Claim
\ref{claim:j}. Indeed, going from $j$ to $j+1$ corresponds to lowering the interface by $\approx 1/L$, in a time $T_{i,j+1}-T_i\approx L$, i.e. corresponds to a negative drift of order $-1/L^2$.  \end{remark}

\begin{proof}
  [Proof of Claim \ref{claim:j}] We proceed by induction on $j$ and we
  observe that for $j=0$ the claim is trivial (it just reduces to
  \eqref{eq:11}, that is a consequence of \eqref{eq:9} which we
  assumed to hold for the value $i$).  We want to prove \eqref{eq:12},
  given the same claim for $j-1$.

Let $V$ be  a shrinking of $U$ by $\epsilon_L^2$, i.e. let
\begin{eqnarray}
  \label{eq:31}
V=U\setminus \cup_{x\in\partial U} B(x,\epsilon_L^2),
  \end{eqnarray}
with $B(x,r)$ the ball of radius $r$ centered at $x$.

\begin{remark}
\label{rem:UV}
 We claim first
of all that it is sufficient to prove \eqref{eq:9} at lattice sites $v\in V$.
Indeed, recall that the height on $\partial U_L$ is always fixed (for
all times) to the boundary height $\varphi_L$.  Since both the height
function $h_t$ and $\bp$ are uniformly $1$-Lipschitz in space and
$|h_t-\bp|\simeq|\varphi_L-\varphi|=O(1/L)$ at the boundary $\partial U_L$ one deduces that,
\emph{deterministically}, $|h_t-\bp|=O(\epsilon_L^2)$
in $U_L\setminus V$. 
On the other hand, 
for $i<\lfloor G/\epsilon_L\rfloor-1$ and $j\le N$
\begin{multline}
\gamma_{i,j}-h_t=\bp-h_t+G-(i+2)\epsilon_L+(1-j/L)\epsilon_L\psi  \\
\ge \bp-h_t +\epsilon_L (1-N/L)\psi_{max}\frac{\psi}{\psi_{max}}\ge  \bp-h_t +
\epsilon_L/\psi_{max}
\end{multline}
since $\psi\ge 2$ and $(1-N/L)\psi_{max}=1+o(1)$.  We have seen that
for the sites within distance $\epsilon_L^2$ from $\partial U$ one has
$|\bp-h_t|=O(\epsilon_L^2)\ll \epsilon_L/\psi_{max}$
deterministically, so the inequality $h_t\le \gamma_{i,j}$ holds
automatically.

In conclusion,  we do not have to worry about 
lattice sites too close to the boundary $\partial U$, see also Remark \ref{rem:ufr}.
\end{remark}

One has
\begin{eqnarray}
  \label{eq:42}
  \bbP(  E_{i,j})\ge \bbP(E_{i,j-1})-\bbP(E_{i,j-1};E_{i,j}^c)\ge
  1-\frac iL-\frac{j-1}{L^3} -\bbP(E_{i,j-1};E_{i,j}^c).
\end{eqnarray}
Next, via a union bound,
\begin{eqnarray}
  \label{eq:43}
  \bbP(E_{i,j-1};E_{i,j}^c)\le \sum_{u\in V}
 \bbP(E_{i,j-1};h_t(u)>\gamma_{i,j}(u)\text{ for some } t\in [T_{i,j},L^5]).
\end{eqnarray}
Since there are $O(L^2)$ sites $u\in V$, it is sufficient to
prove
\begin{eqnarray}
  \label{eq:44}
  \bbP(E_{i,j-1};h_t(u)>\gamma_{i,j}(u)\text{ for some } t\in
  [T_{i,j},L^5])\le 1/L^6
\end{eqnarray}
to deduce \eqref{eq:12}.

\smallskip

 Let the stepped monotone interface $\gamma^{(L)}_{i,j-1}$ be a
discretization of $\gamma_{i,j-1}$ such that 
\begin{eqnarray}
  \label{eq:46}
  \gamma_{i,j-1}^{(L)}> \gamma_{i,j-1}
\end{eqnarray}
(strict inequality) and $\{\hat h_t\}_{t\ge T_{i,j-1}}$ be the
Markov dynamics with (random) initial condition $h_{T_{i,j-1}}$ at time $T_{i,j-1}$, and such
that:
\begin{itemize}
\item if $h_{T_{i,j-1}}\le \gamma^{(L)}_{i,j-1}$, then $\hat h_t$ is the
  dynamics with ceiling $\gamma^{(L)}_{i,j-1}$;

\item if instead $h_{T_{i,j-1}}\not\le \gamma^{(L)}_{i,j-1}$, then
  $\hat h_t=h_{T_{i,j-1}}$ for every $t\ge T_{i,j-1}$.
\end{itemize}
Note that, on the event $E_{i,j-1}$, one has $h_{T_{i,j-1}}\le
\gamma^{(L)}_{i,j-1}$ and moreover the two dynamics $h_t$  and $\hat
h_t$ can be coupled so that they exactly coincide in the time interval
$[T_{i,j-1},L^5]$. In fact, from the definition of dynamics with ceiling (Section \ref{sec:monoton}) the two dynamics coincide until the first
time $\tau$ when $h_\tau(v)=\gamma^{(L)}_{i,j-1}(v)$ for some $v$;
on the event
$E_{i,j-1}$ one has
$h_t\le \gamma_{i,j-1}<\gamma^{(L)}_{i,j-1}$ up to time $L^5$ time and therefore $\tau\ge L^5$. 
Therefore, the probability in \eqref{eq:44} can be upper bounded by 
\begin{eqnarray}
  \label{eq:45}
\bbP(\hat h_t(u)>\gamma_{i,j}(u) \text{ for some } t\in
  [T_{i,j},L^5] | \hat h_{T_{i,j-1}}=h_{T_{i,j-1}}\le \gamma^{(L)}_{i,j-1}
).
\end{eqnarray}

Next, we want to reduce from the dynamics $\hat h_t$ in the whole $U_L$ to a dynamics where only the height function in a much smaller domain $D_u$ evolves.
Given a lattice site $u\in V$ let 
$D_u$ be a disk\footnote{
To be precise, $D_u$ should be a discrete domain; take $D_u$  as the union of triangles in $\mathcal
T_L$ contained in such a disk. For lightness of exposition,
we will overlook this minor detail and just call $D_u$ a ``disk''.} of radius $ L^{-1/2+\delta/100}$ centered
at $u$: from the definition \eqref{eq:31} of $V$, we see that the disk $D_u$ is entirely contained
  in $U$, since $L^{-1/2+\delta/100}\ll\epsilon_L^2$.

  We start by observing that, by monotonicity, since we want to upper
  bound \eqref{eq:45}, we are allowed to change the random
  configuration $\hat h_{T_{i,j-1}}$ at time $T_{i,j-1}$ to the
  deterministic configuration $\gamma_{i,j-1}^{(L)}\ge\hat h_{T_{i,j-1}} $, and to freeze
  $\hat h_t(v)$ to $\gamma_{i,j-1}^{(L)}(v)$ for times $t\ge
  T_{i,j-1}$ and sites $v$ outside $D_u$. In words, we are pinning the
  height function to the ceiling outside $D_u$. Again by monotonicity,
  we impose that the evolution $\hat h_t$ has a ``floor'' constraint
  $\hat h_t(v)\ge \gamma^{(L)}_{i,j-1}-L^{-1+\delta/40}$. We still
  call $\{\hat h_t\}_{ t\ge T_{i,j-1}}$ the dynamics after these two
  modifications, and we denote $\hat \pi_{i,j-1}$ its equilibrium
  measure (it is the uniform measure on stepped monotone interfaces in
  $D_u$, with boundary height $\gamma_{i,j-1}^{(L)}|_{\partial D_u}$ ,
  ceiling $\gamma_{i,j-1}^{(L)}$ and floor
  $\gamma^{(L)}_{i,j-1}-L^{-1+\delta/40} $).
\begin{remark}
  \label{rem:ufr} We have used crucially that $D_u\subset U$, more precisely that boundary sites on $\partial D_u$ are in $U_L$ to say that, on the event
$E_{i,j-1}$,  $h_t(v)< \gamma^{(L)}_{i,j-1}(v)$ for $v\in \partial D_u$ and $t\in [T_{i,j-1},L^5]$. If we had to consider points much closer to the boundary,
we would have to take a disk $D_u$ of smaller diameter (so that it fits in
$U_L$) and then the proof of Proposition \ref{th:equilibrio} below would fail.
\end{remark}

Let us assume for the moment the  following equilibrium estimate:
  \begin{proposition}
\label{th:equilibrio}
    Let $j\le N$ and  $\hat\pi_{i,j-1}$ be   as above.
If $\xi$ in \eqref{eq:psi} is smaller than some $\xi_0>0$ (that is independent of $i,j,L$) then
\begin{eqnarray}
  \label{eq:15}
 \hat\pi_{i,j-1}\left[ h(u)> \gamma_{i,j-1}(u)-L^{-1+\delta/60}\right]\le L^{-20}.
\end{eqnarray}
  \end{proposition}
  Let us conclude the proof of the step $j-1\to j$, given Proposition
\ref{th:equilibrio}.
By Lemma \ref{lemma:Wilson}, since the graph-distance diameter of $D_u$ is $O(L^{1/2+\delta/100})$ and the distance between floor $\gamma^{(L)}_{i,j-1}-L^{-1+\delta/40}$
and ceiling $\gamma^{(L)}_{i,j-1}$ is $L^{-1+\delta/40}$, the  mixing time of the  dynamics $\hat h_t$
  is \[O(L^{2(1/2+\delta/100)}L^{2\delta/40}(\log L)^2)\le L^{1+\delta/8}\ll
  T_{i,j}-T_{i,j-1}=L^{1+\delta/4}.\] Therefore, at time $T_{i,j}$ equilibrium $\hat\pi_{i,j-1}$ has been reached, up to a negligible
  variation distance error $O(\exp(-L^{\delta/8}))$.
  As a consequence,
  \begin{multline}
    \label{eq:40}
    \bbP(\hat h_t(u)> \gamma_{i,j}(u)\text{ for some } t\in [T_{i,j}, L^5])\\=
\bbP^{\hat \pi_{i,j-1}}(\hat h_t(u)> \gamma_{i,j}(u)\text{ for some } t\in [T_{i,j},
L^5])+O(\exp(-L^{\delta/8}))\\
\le L^7\hat\pi_{i,j-1}(h(u)> \gamma_{i,j}(u)) +O(\exp(-L^{\delta/8})).
  \end{multline}
In the second line,
$\bbP^{\hat \pi_{i,j-1}}$ denotes the law of the modified dynamics, with
initial condition sampled from the equilibrium distribution $\hat
\pi_{i,j-1}$; in the third line  we used the standard fact that, for a
continuous-time homogeneous Markov
chain $X_t$ with invariant measure $\pi$ and any event $A$,
\[\bbP^\pi(\exists t\in [a,b]: X_t\in A)\le |b-a|M\pi(X\in A)\]
with $M$ the average number of updates per unit time. In our case
$M$ can be bounded by the number of lattice sites in $D_u$ (since each site has a mean-one Poisson clock), which is much smaller than  $L^2$.

Note that 
\begin{eqnarray}
  \label{eq:16}
\gamma_{i,j-1}(u)-L^{-1+\delta/60}<\gamma_{i,j}(u):
\end{eqnarray}
just recall \eqref{eq:13} and observe that
$L^{-1+\delta/60}>\epsilon_L\psi(u)/L$. As a consequence, from Proposition
\ref{th:equilibrio}, 
\begin{eqnarray}
  \label{eq:47}
  \hat\pi_{i,j-1}(h(u)>\gamma_{i,j}(u))\le  \hat\pi_{i,j-1}(h(u)>\gamma_{i,j-1}(u)-L^{-1+\delta/60})\le L^{-20}
\end{eqnarray}
and \eqref{eq:44} follows from \eqref{eq:45} and \eqref{eq:40}. The proof of Claim \ref{claim:j} is concluded.
   \end{proof}

\begin{remark} 

The common points with the proof of \cite[Theorem 2]{CMT}
(that is the analog of Theorems \ref{th:pratica2} and \ref{th:praticamente} in the particular case of flat macroscopic shape) are the pervasive use of monotonicity and the 
idea of employing Lemma \ref{lemma:Wilson} on mesoscopic domains of size 
slightly larger than $L^{-1/2}$ (specifically, $L^{-1/2+\delta/100}$ here). 
\end{remark}

\section{Proof of Proposition \ref{th:equilibrio}}

\label{sec:proofeq}
Since $i$ and $j$ are fixed in this section, we
write for simplicity of notation \[\gamma=\gamma_{i,j-1}=G-(i+2)\epsilon_L+\bar\phi+\kappa\epsilon_L\,
\psi\]
with $\kappa=(1-(j-1)/L)$.
Recall that $(j-1)<N$, with $N$ as in Definition \ref{def:psi}, so that
\begin{eqnarray}
  \label{eq:kappa}
  \kappa\in [(1+O(1/L))/\psi_{max}, 1].
\end{eqnarray}
In fact, we will need that $\kappa$ is bounded away from zero uniformly in $L,i,j$.

Recall that $u=(x,y)\in V$, $D_u\subset U$ is a disk centered at $u$, of
radius $ L^{-1/2+\delta/100}$ and that $\hat \pi:=\hat \pi_{i,j-1}$ is the uniform measure over
stepped interfaces in $D_u$, with boundary condition
$\gamma^{(L)}_{i,j-1}|_{\partial D_u}$, ceiling $\gamma^{(L)}_{i,j-1}$ and floor $\gamma^{(L)}_{i,j-1}-L^{-1+\delta/40}$.
Since we want to prove the \emph{upper bound} \eqref{eq:15}, we can
by monotonicity remove the ceiling.

The floor cannot be removed by monotonicity. However:
\begin{lemma}
\label{lemma:floor}
  Let $\tilde \pi$ be obtained from $\hat \pi$ by eliminating the
  floor constraint $h\ge \gamma^{(L)}_{i,j-1}-L^{-1+\delta/40}$. We have $\|\hat \pi-\tilde\pi\|=O(L^{-n})$ for any given $n$.
\end{lemma}
This will be proven at the end of this section. It is clear that it is
sufficient to prove Proposition \ref{th:equilibrio} with $\hat\pi$
replaced by $\tilde\pi$.
\medskip

Call $\nabla
\bp(u)=(\partial_x{\bar\phi}(u), \partial_y{\bar\phi}(u))$ the slope
of $\bar\phi$ at $u$ and $ H^{\bar\phi}(u)$ its $2\times 2$ Hessian
matrix. These are well-defined objects, since $\bp$ is non-extremal
and therefore infinitely differentiable, see Section \ref{sec:macro}.
Similarly, call \[ \nabla
\gamma(u)=(\partial_x\gamma(u),\partial_y\gamma(u))=\nabla{\bar\phi}(u)+\kappa\epsilon_L\,\nabla{\psi}(u)\]
and
\[H^\gamma(u)=H^{\bar\phi}(u)+\kappa\epsilon_L\, H^{\psi}(u)\] the slope and Hessian of
$\gamma$. The argument $u$ will be omitted unless needed for clarity. Note that $\nabla\gamma$ and $H^\gamma$ do \emph{not} in general
satisfy \eqref{eq:PDE}, i.e.  in general
\begin{eqnarray}
  \label{eq:nottru}
{\bf a}(\nabla \gamma)\cdot H^\gamma:=\sum_{i,j=1,2}a_{ij}(\nabla\gamma)H^\gamma_{ij}\ne0, 
\end{eqnarray}
because $\gamma$ is \emph{not} the equilibrium
shape with some boundary height. In other words,
$(\nabla\gamma,H^\gamma)$ is \emph{not} the local structure of any
macroscopic shape.

\begin{remark}
  One has
\begin{eqnarray}
  \label{eq:21}
\nabla\psi=-\frac1\xi(\,e^{x/\xi},e^{y/\xi})
\end{eqnarray}
and
\begin{eqnarray}
\label{eq:norma}
\|\nabla\gamma-\nabla\bp\|= \kappa \epsilon_L \|
\nabla\psi\|\le \mathcal K(\xi,\kappa,\epsilon_L,u):=\sqrt2\frac{\kappa\epsilon_L}\xi\,
e^{\max(x,y)/\xi},
\end{eqnarray}
which is $o(1)$ when $L\to\infty$. In particular, since $\nabla \bp$ is uniformly bounded away from $\partial \tri$, so is $\nabla \gamma$.
\end{remark}

Define then the $2\times 2$ matrix 
\begin{eqnarray}
  \label{eq:zeta}
  {\bar H}
=H^{\bar\phi} - \frac{H^{\bar\phi}\cdot{\bf a}(\nabla\gamma) }{H^{\psi}\cdot{\bf a}(\nabla\gamma)}H^\psi
\end{eqnarray}
with $H^\psi$ defined in \eqref{eq:8}.
The denominator is non-zero always: recall from Section
\ref{sec:macro} that the diagonal elements $a_{ii}$ given in
\eqref{eq:aij} are positive away from $\partial \tri$ and 
that $H^\psi$ is diagonal, with negative
diagonal entries, see \eqref{eq:8}.
Remark also that,
in contrast with
\eqref{eq:nottru},
\begin{eqnarray}
  \label{eq:questasi}
 {\bf a}(\nabla \gamma)\cdot \bar H=0.
\end{eqnarray}
In other words, we have added the right correction to $H^\gamma$ so that
$(\nabla\gamma, \bar H)$ is a local structure of a macroscopic shape.
We further remark that
\begin{eqnarray}
  \label{eq:diffe}
  H^\gamma-\bar H=\left(\kappa\epsilon_L+ \frac{H^{\bar\phi}\cdot{\bf a}(\nabla\gamma) }{H^{\psi}\cdot{\bf a}(\nabla\gamma)}\right)H^\psi=
\kappa\epsilon_L(1+O(\xi))H^\psi.
\end{eqnarray}
To see this, observe that for the denominator
\begin{eqnarray}
  \label{eq:1}
  |H^{\psi}\cdot {\bf a}(\nabla\gamma)|\ge \frac c{\xi^2}e^{\max(x,y)/\xi}
\end{eqnarray}
with $c=\min_{(x,y)\in U}\min(a_{11}(\nabla\gamma),a_{22}(\nabla\gamma))>0$; for 
the numerator, 
\begin{eqnarray}
  \label{eq:forthe}
  H^{\bp}\cdot {\bf a}(\nabla\gamma)=
  H^{\bp}\cdot {\bf a}(\nabla\bp)+
  H^{\bp}\cdot ({\bf a}(\nabla\gamma)-{\bf a}(\nabla\bp))\\=0+ H^{\bp}\cdot ({\bf a}(\nabla\gamma)-{\bf a}(\nabla\bp)).
\end{eqnarray}
From \eqref{eq:norma} and \eqref{eq:1}
one deduces
\begin{eqnarray}
  \label{eq:22}
  \frac{ H^{\bp}\cdot {\bf a}(\nabla\gamma)}{H^{\psi}\cdot {\bf a}(\nabla\gamma)}=O\left(
\kappa \epsilon_L\,\xi
\right)
\end{eqnarray}
and \eqref{eq:diffe} follows.

\medskip

Let \[
z(u)=(\partial_x \gamma(u),\partial_y\gamma(u),\bar H_{11}(u),\bar
H_{12}(u))\in\mathcal A
\]
and define
\begin{eqnarray}
  \label{eq:ws}
w(u)=(A(u),B(u),X(u),Y(u))=f^{-1}(z(u))\in W,
\end{eqnarray}
with $f:W\mapsto \mathcal A$ the function of Theorem \ref{th:crazy1}.
That is, recalling also \eqref{eq:questasi}, $(\nabla \gamma(u),\bar H(u))$ is the 
local structure of the macroscopic shape $\bp_{A(u)B(u)C(u)}$ in the hexagon $\varhexagon_{A(u)B(u)C(u)}$, with $C(u)=1-A(u)-B(u)$, at the point $(X(u),Y(u))\in \Circle_{A(u)B(u)C(u)}$. 
\begin{remark}
  Observe that, since $\bp$ is non-extremal, the closure $K$ of the set 
\[
K'=\{(\partial_x \bp(u),\partial_y\bp(u),\partial^2_x\bp(u),
\partial^2_{xy}\bp(u)), \; u\in U\}
\]
is a compact subset of the open set $\mathcal A$. By Proposition  \ref{th:bordobordo}, 
$f^{-1}(K)$ is a compact subset of the open set $W$.
Next, note that
\begin{eqnarray}
\label{eq:zuzu}
\sup_{u\in U}\|z(u)-(\partial_x\bp(u),\partial_y\bp(u), \partial^2_x\bp(u),\partial^2_{xy}\bp(u))\|=o(1)
\end{eqnarray}
when $L\to\infty$.
Indeed, $\|\nabla\gamma-\nabla\bp\|$ was bounded in \eqref{eq:norma}, while
\begin{eqnarray}
\label{eq:normaz}
  \|\bar H-H^{\bar\phi}\|=\frac{|H^{\bar\phi}\cdot {\bf a}(\nabla\gamma)|}{
|
(H^\psi/\|H^\psi\|)\cdot {\bf a}(\nabla\gamma)
|
}=O(\sup_{u\in U}\mathcal K(\xi,\kappa,\epsilon_L,u))=o(1):
\end{eqnarray}
for the numerator see \eqref{eq:forthe} and \eqref{eq:norma}, while
the denominator is lower bounded by 
\[
(1/\sqrt 2)\times\min_{(x,y)\in U}\min(a_{11}(\nabla\gamma),a_{22}(\nabla\gamma))
\]
which is positive, as discussed before.

As a consequence of \eqref{eq:zuzu} and of Proposition
\ref{th:bordobordo}, we see that $\{w(u) , u\in U\}$ is contained in a
compact subset of $W$. In other words,  uniformly in $u\in U$, $A(u), B(u), C(u)$ are bounded
away from zero and $(X(u), Y(u))$ is bounded away from $\partial
\Circle_{A(u)B(u)C(u)}$. This remark is
important to guarantee that, when one applies Theorem \ref{th:Petrov},
estimates one obtains are uniform with respect to $u$. Uniformity will
be not be recalled explicitly later.
\end{remark}

From now on, for lightness of notation we remove the argument $u$ from
$A(u), B(u), C(u)$.
 Call, as in Section \ref{sec:hexa}, $ \pi_L^{ABC}$ the uniform measure in the 
hexagon $\varhexagon_{A_LB_LC_L}$ (the discretization of
$\varhexagon_{ABC}$) and $\bar \phi_{ABC}$ the macroscopic shape. 
Translate $\varhexagon_{A_LB_LC_L}$ in the $P_{111} $ plane so that $u=(x,y)$ and
$(X(u),Y(u))$ coincide, and add a suitable global constant to the
boundary height $\varphi_{a_Lb_Lc_L}$ on $\partial \varhexagon_{A_LB_LC_L}$, so that
\begin{eqnarray}
  \label{eq:20}
\pi_L^{ABC}(h(u))=\gamma(u).
\end{eqnarray}

We are at last in a position to prove the claim of Proposition \ref{th:equilibrio}, i.e. that
\begin{eqnarray}
\label{eq:atlast}
  \tilde \pi(h(u)>\gamma(u)-L^{-1+\delta/60})\le L^{-20}
\end{eqnarray}
(recall that $\tilde \pi$ differs from $\hat \pi$ in that the floor
constraint $h\ge \gamma^{(L)}_{i,j-1}-L^{-1+\delta/40}$ has been removed).

Recall that the boundary condition on  $\partial  D_u$ (that is a disk of radius $L^{-1/2+\delta/100}$ centered at $u$), is
$\gamma^{(L)}_{i,j-1}$, that is a discretization of
$\gamma=\gamma_{i,j-1}$.  We have for $v\in D_u$
\begin{multline}
\label{gga}
  \gamma^{(L)}_{i,j-1}(v)=\gamma(v)+O(1/L)\\=\gamma(u)+(v-u)\cdot\nabla\gamma(u)+
\frac12(v-u)\cdot H^\gamma(u)\cdot(v-u)+O(1/L)
\end{multline}
where we used $\|u-v\|\sim L^{-1/2+\delta/100}$ and smoothness
of $\gamma$ to ignore higher-order terms in the Taylor expansion of $\gamma$ around $u$.
Next, write $H^\gamma(u)=\bar H(u)+(H^\gamma(u)-\bar H(u))$, with $\bar H$ defined in \eqref{eq:zeta}. From \eqref{eq:diffe} and \eqref{eq:8} we have for $\xi$ small enough
\begin{equation}
  (v-u)\cdot (H^\gamma(u)-\bar H(u))\cdot (v-u)\le -C_1\epsilon_L \|u-v\|^2,
\end{equation}
with $C_1(\xi,\kappa)$ a strictly positive constant, independent of $L$
(recall from \eqref{eq:kappa} that $\kappa $ is bounded away from zero).
Therefore, 
\begin{equation}
\label{eq:there}
    \gamma^{(L)}_{i,j-1}(v)\le \gamma(u)+(v-u)\cdot\nabla\gamma(u)+\frac12(v-u)\cdot \bar H(u)\cdot(v-u)-C_1\epsilon_L \|u-v\|^2+O(1/L).
\end{equation}
Now recall that $\nabla \gamma(u)$ and $\bar H(u)$ are also the
gradient and Hessian at $u$ of $\bar \phi_{ABC}$, the macroscopic
shape in the hexagon $\varhexagon_{ABC}$. A second-order Taylor
expansion and \eqref{eq:20} give then
\begin{multline}
   \gamma^{(L)}_{i,j-1}(v)\le (\gamma(u)-\bp_{ABC}(u))+\bp_{ABC}(v)-C_1\epsilon_L\|u-v\|^2+O(1/L)\\
=(\pi_L^{ABC}(h(u))-\bp_{ABC}(u))+\bp_{ABC}(v)-C_1\epsilon_L \|u-v\|^2+O(1/L).
\end{multline}
Thanks to Theorem \ref{th:Petrov} we have $(\pi_L^{ABC}(h(u))-\bp_{ABC}(u))=(\pi_L^{ABC}(h(v))-\bp_{ABC}(v))+O(1/L)$
and finally 
\begin{eqnarray}
\label{eq:tosumma}
 \gamma^{(L)}_{i,j-1}(v)\le   \pi_L^{ABC}(h(v))-C_1\epsilon_L \|u-v\|^2+O(1/L).
\end{eqnarray}

Taking $v\in\partial D_u$, so that $\|u-v\|\sim L^{-1/2+\delta/100}$,
we deduce that the boundary height $\gamma^{(L)}_{i,j-1}|_{\partial
  D_u}$ in the measure $\tilde \pi$ is lower than the
function \[\partial D_u\ni v\mapsto \pi_L^{ABC}(h(v))-C_2 \epsilon_L
L^{-1+\delta/50}.\] An immediate application of Proposition
\ref{th:Petrov2} gives that the $\tilde \pi$-probability
that \[h(u)>\gamma(u)-L^{-1+\delta/60}=
\pi_L^{ABC}(h(u))-L^{-1+\delta/60}\] is $O(L^{-n})$ for any given
$n$. Choosing $n>20$ we get the desired result \eqref{eq:atlast}.
\qed

\begin{proof}[Proof of Lemma \ref{lemma:floor}]
  Observe that, with our usual notations, $\tilde
  \pi=\pi_{D_u}^{\gamma^{(L)}_{i,j-1}}$ and
\[
\hat\pi=\tilde\pi(\cdot\;|\;\cdot\ge \gamma^{(L)}_{i,j-1}-L^{-1+\delta/40}),
\]
from which it is immediate to deduce 
\[
\|\hat\pi-\tilde\pi\|\le
\frac{\tilde\pi(h\not\ge
  \gamma^{(L)}_{i,j-1}-L^{-1+\delta/40})}{1-\tilde\pi(h\not\ge \gamma^{(L)}_{i,j-1}-L^{-1+\delta/40})},
\]
where $h\not\ge \gamma^{(L)}_{i,j-1}-L^{-1+\delta/40}$ (violation of
the floor constraint) is the event that there exists $v\in D_u$ such that $h(v)<\gamma^{(L)}_{i,j-1}-L^{-1+\delta/40}$.

It is not hard to see that $\tilde\pi(h\not\ge \gamma^{(L)}_{i,j-1}-L^{-1+\delta/40})$ is $O(L^{-n})$ for any given $n$. Indeed, analogously to \eqref{eq:tosumma} 
one has 
that, for
$v\in \partial D_u$,
\begin{eqnarray}
\label{eq:tosumma2}
 \gamma^{(L)}_{i,j-1}(v)\ge  \pi^{ABC}_L(h(v))-C_3 \epsilon_L L^{-1+\delta/50}.
\end{eqnarray}
From Proposition \ref{th:Petrov2} one deduces that, except with $\tilde\pi$-probability $O(L^{-n})$,
\[
h(v)\ge\pi^{ABC}_L(h(v))-(1/2)L^{-1+\delta/40}
\quad\text{ for every }\quad v\in D_u.
\]
On the other hand,
from \eqref{eq:tosumma} we see that for every $v\in D_u$
\[
\gamma^{(L)}_{i,j-1}(v)-L^{-1+\delta/40}\le \pi_L^{ABC}(h(v))-(3/4)L^{-1+\delta/40}.
\]
\end{proof}

\section{Proof of Corollary \ref{cor:quasifrozen}}

Let 
\begin{eqnarray}
  \label{eq:epn}
  \epsilon^{(n)}=\max_{x\in\partial U}|\varphi(x)-\varphi^{(n)}(x)|
\end{eqnarray}
and set \[\tilde\varphi^{(n)}=\varphi^{(n)}-\min_{x\in\partial U}(\varphi^{(n)}(x)-\varphi(x))\]
(we are just adding a constant to the boundary height $\varphi^{(n)}$).
Note that 
\begin{eqnarray}
  \varphi\le \tilde\varphi^{(n)}\le \varphi+2\epsilon^{(n)}
\end{eqnarray}
so that, if $\tilde\phi^{(n)}$ is the macroscopic shape in $U$ with boundary height
$\tilde\varphi^{(n)}$, one has
\begin{eqnarray}
\label{eq:confr}
  \bp\le \tilde\phi^{(n)}\le \bp+2\epsilon^{(n)}.
\end{eqnarray}
If $h_t^\eta$ and $\tilde h_t^\sigma$ denote respectively the dynamics with boundary heights $\varphi,\tilde\varphi^{(n)}$ and initial conditions $\eta,\sigma$, one
has by monotonicity
\begin{multline}
\label{eq:thirdline}
  \max_\eta\sup_{t>c(\delta)L^{2+\delta}}
\bbP(\exists v\in U_L: h_t^\eta(v)> \bp(v) +3\epsilon^{(n)})\\\le
 \max_\sigma\sup_{t>c(\delta)L^{2+\delta}}\bbP(\exists v\in U_L:\tilde h_t^\sigma(v)> \bp (v)+3\epsilon^{(n)})\\
\le
\max_\sigma\sup_{t>c(\delta)L^{2+\delta}}\bbP(\exists v\in U_L:\tilde h_t^\sigma(v)> \tilde\phi^{(n)}(v)+\epsilon^{(n)})
\end{multline}
where we used \eqref{eq:confr} in the second step.
Theorem \ref{th:pratica2} says that the third line of \eqref{eq:thirdline}
is some function $f(n,L)$ that vanishes as $L\to\infty$.
It is then standard that one can choose some sequence $n(L)$ that diverges
as $L\to\infty$ for which $f(n(L),L)$ still converges to zero.
In conclusion, setting $\epsilon_L:=\epsilon^{(n(L))}$, that tends to zero as $L\to\infty$, we have proven
\begin{eqnarray}
    \max_\eta\sup_{t>c(\delta)L^{2+\delta}}
\bbP(\exists v\in U_L: h_t^\eta(v)> \bp(v) +3\epsilon_L)=o(1)
\end{eqnarray}
which is ``half'' of the claim \eqref{eq:provaw}. The other bound can be obtained similarly.
\appendix

\section{Proof of Eq. \eqref{eq:ave}}
\label{app}

\subsection{Mean height}

We have to compute the average height
difference between two lattice points $u$ and $v$ in $\mathcal T_L\cap \Circle_{abc}$. Recall 
Remark \ref{rem:hl}: the height difference is directly related to the
number of lozenges crossed by a lattice path from $u$ to $v$ (the
height difference is independent of the chosen path). We will assume that the vector $u-v$ is along one of
the three lattice directions of $\mathcal T_L$ (in the general case,
we can always reduce to this situation by
choosing the path from $u$ to $v$ as a concatenation of a $L$-independent number
of straight paths along these directions) and by symmetry we 
consider only the case where $u-v$ is in the vertical direction, with
$v$ above $u$.

To avoid a plethora of $\lfloor \cdot \rfloor$, let us
 assume that \[A=aL,B=bL, C=cL\] are even
integers.
Recall our choice of (non-orthogonal) coordinates $(x,y)$ for points in $\mathcal T_L$. In this Appendix, to fit better the notation of \cite{Petrov1}, it is
convenient to translate the origin of the coordinates (that used to be in the center of the hexagon until now) in such a way that the center of the hexagon
$\varhexagon_{a_Lb_Lc_L}$ has coordinates 
\[
(-a-(b+c)/2+1/(2L),-(3/2)a-c-b/2+1/(2L)).
\]
See Figure \ref{fig:coordonnees_lozenges}.  Note that we translated the origin by a half-integer number of
lattice steps in both directions $x $ and $y$. Given a vertical edge
in $\mathcal T_L$,
let us label it with the coordinates $(x,y)$ (that are now integers times $1/L$) of its
mid-point.
Again to stay closer to the notations of \cite{Petrov1}, here we assume that $a+c=1$, instead of our usual normalization $a+b+c=1$.

Recall that $u$ and $v$ are points in $\mathcal T_L$ related by a
vertical segment and let $p_i,i=0,\dots,n-1$ be the vertical lattice
edges composing such segment (labeled say from below). Clearly,
$n=L|v-u|$
(because each edge of $\mathcal T_L$ has length $1/L$).
In this case from Remark \ref{rem:hl} we see that
\begin{eqnarray}
  \label{eq:2}
\pi_L^{abc}(h(v))-\pi_L^{abc}(h(u))=\frac1L\sum_{i=0}^{n-1}\left[1-\pi_L^{abc}({\bf
    1}_{p_i})\right]  
\end{eqnarray}
with  ${\bf 1}_{p}$ the indicator function that the vertical edge $p$ crosses a horizontal lozenge. 
We will prove the following estimate (which is a special case of
estimates proved in \cite[Section 7]{Petrov1}, except for the explicit
control of the error term):
\begin{proposition}\label{prop:noyau_horizontal}
  Let $p=(x,y)$ be  a vertical edge contained in  $ \Circle_{abc}$. 
One has
\begin{eqnarray}
  \label{eq:kinf}
 \pi_L^{abc}({\bf 1}_{p})= \Pi_\infty(x,y) +O(1/L):=\frac{1}{2\pi i} \int_{\bar{\rm
      w}_c}^{{\rm w_c}}
  \frac{(1+y-x)}{(z+x)(z +y +1)}d z +O(1/L)  
\end{eqnarray}
where:
\begin{itemize}
\item  the $O(1/L)$ error is uniform for $(x,y)$ in  compact subsets of
  $\Circle_{abc}$;
\item  ${\rm w_c}={\rm w_c}(x,y,a,b,c)$ is the unique 
  non-real critical point (w.r.t. $w$) of the
  function $S$, defined in equation \eqref{eq:def_S}, in the upper half complex plane and
  $\bar {\rm w}_c$ is the complex conjugate of ${\rm w}_c$;
\item the contour of integration in the complex plane intersects the real axis to the right
  of both poles.
\end{itemize}
\end{proposition}
Now we can plug   \eqref{eq:kinf} into
\eqref{eq:2}. After summation on $i$, the error term $O(1/L)$ gives an error term
$O(n/L^2)=O(|u-v|/L)$. The main term
$1-\Pi_\infty$ gives as dominant term the line integral
\begin{eqnarray}
  \label{eq:3}
\int_{\mathcal C(u,v)} (1-\Pi_\infty)\,ds
\end{eqnarray}
(with $\mathcal C(u,v)$ the straight path from $u$ to $v$)
plus again an error $O(|u-v|/L)$ by Riemann approximation.

Finally we do \emph{not} need to check that the line integral equals
$\bar\phi_{abc}(v) - \bar\phi_{abc}(u)$: this follows from Theorem
\ref{th:CLP} on the existence of the limiting shape.

\subsection{Correlation kernel}

\begin{wraptable}{R}{0.5\textwidth}
      \begin{center}
\begin{tabular}{| c| c |}
\hline\hline
Our notations & Notations from \cite{Petrov1}\\
\hline\hline
-1/2 & $A_1$\\
$A$ & $B_1-A_1$\\
$B$ & $A_2-B_1$\\
$C$ & $B_2-A_2$\\
$x$ & $-x$\\
$y$ & $-x-n$ \\
$L$ & $N$\\
\hline
      \end{tabular}
      \end{center}
\caption{The correspondence between the two sets of notations, in the case where in \cite{Petrov1} the origin is chosen such that $A_1=-1/2$.}
\label{tab:cambio}
\vspace{0cm}
\end{wraptable} 

The proof of Proposition \ref{prop:noyau_horizontal} is essentially
identical to the proof of Theorem 2 of \cite[Section 7]{Petrov1},
except that we keep track more precisely of the size of
errors. Actually, Theorem 2 of  \cite{Petrov1} considers a much more general situation: first of all, the domain to be tiled is not simply a hexagon  but a more general polygonal shape (the hexagon being a particular case).
Secondly, Theorem 2 of \cite{Petrov1} allows  to get the 
asymptotics of the probability of any event involving a fixed number $m$ of lozenges
(say, the probability of the event ${\bf 1}_{p_1}\cdots{\bf 1}_{p_m}$). 
For simplicity of exposition, we will however restrict ourselves to the  hexagonal region
$\varhexagon_{a_Lb_Lc_L}$ (which corresponds to the polygonal shape of \cite{Petrov1} with the choice $k=2$ there) and to 
the observable ${\bf 1}_{p}$ we are interested in. 

\smallskip

\begin{wrapfigure}{R}{0.5\textwidth}
\begin{center}
 \includegraphics[width=0.4\textwidth]{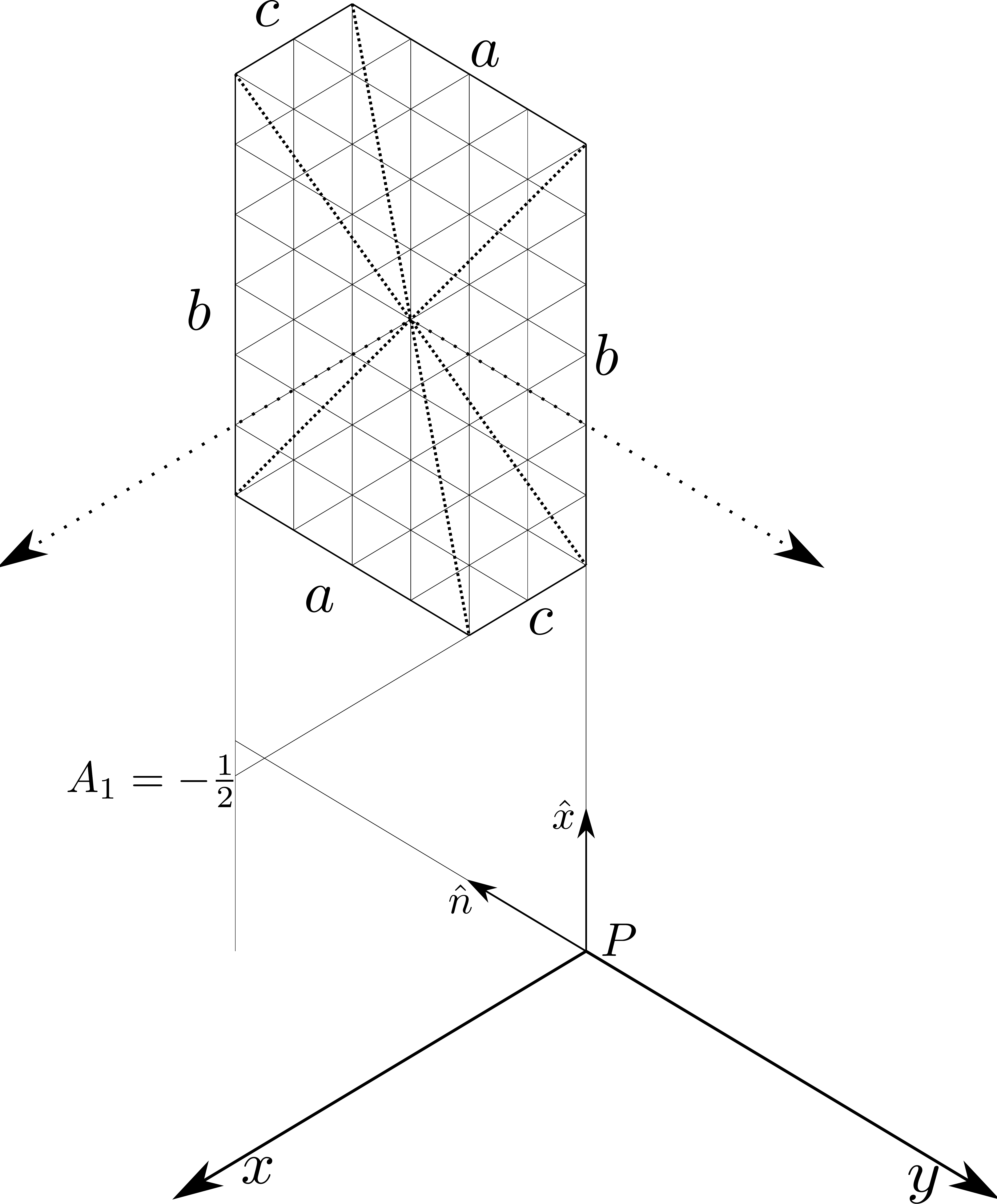}
 \caption{The coordinates we used until now (dotted axes) had the
   origin at the center of the hexagon. The coordinates we use in the
   Appendix are centered in $P$. The axes $\hat n,\hat x$ correspond
   to the the coordinates used in \cite{Petrov1}.}
\label{fig:coordonnees_lozenges}
\end{center}
\vspace{0cm}
\end{wrapfigure}

Let us introduce the Pochhammer symbol $(y)_m = y(y+1)\ldots(y+m-1)$.
The probability that the vertical edge  $p=(x,y)$ crosses
a horizontal lozenge can be read from formulas (2.5) and (2.6) of
\cite{Petrov1}: it is given by a double integral in the complex
plane, with a $L$-dependent but explicit kernel. In our language, we have:
\begin{theorem}\label{thm:noyau_exact}
  Let $p=(x,y)$ be a vertical edge in the hexagon.  Define $X=L x,
  Y=Ly$ ($X,Y$ are automatically integers with the present convention
  for coordinates).  One has
\begin{multline}
  \label{eq:losangone}
  \pi_L^{abc}({\bf 1}_{p})=\frac
  {(L-X+Y)}{(2i\pi)^2} \\
\times\oint dZ \oint dW \frac{(Z+X+1)_{L-X+Y-1}}{(W+X)_{L-X+Y+1}} \frac{1}{W-Z}  \frac{(-W)_{A}}{(- Z)_{A} } \frac{(A+B-W)_{C}}{(A+B - Z)_{C} }
\end{multline}
where the integration contours on $Z$ and $W$ have to chosen such that:
\begin{itemize}
  \item The $Z$ contour runs counter-clockwise and includes the integer real points $-X, -X+1, ..., A+B+C$ and no other integer point of the real line;
  \item The $W$ contour runs counter-clockwise; it contains the $Z$ contour and also integer
    points $-Y-L,-Y-L+1,\dots,-X$ (it may contain other integer
    points). 
\end{itemize}
\end{theorem}

We warn the reader who would like to find this formula in \cite{Petrov1} that  conventions in
\cite{Petrov1} are  different from ours: the correspondence between the
two sets of notations is conceptually trivial but a bit tricky as we
need to rotate the hexagon and make an affine transformation to go
from one setting to the other (compare for instance the shape of
lozenges in our Figure \ref{fig:lozenges} and in Figure 3 of \cite{Petrov1}). We will not give details on the
transformation but we summarize the correspondence in Table \ref{tab:cambio}.

\subsection{Contour changes}

In this section we use upper-case letters for lengths and coordinates 
proportional to $L$ 
and lower-case letters for the corresponding rescaled variables of order $1$.

First we change the scale by the change of variable $Z \rightarrow Z/L=z$ and $W \rightarrow W/L=w$ and we reorder the resulting $L$-dependent factor to get:
\begin{eqnarray}
  \label{eq:4}
   \pi_L^{abc}({\bf 1}_p) = \frac{1}{(2 \pi i)^2} \oint \oint \frac{1}{w-z} \frac{(1+y-x)}{(w+x)(w+y +1)} \frac{P_L(w, x,y)}{P_L(z,x,y)}dz\,dw 
\end{eqnarray}
 where 
\[
  P_L ( z, x, y ) = \frac{(L+L y-L x-1)!}{(Lz +L x +1)_{L+Y-X-1}} (- Lz)_{A} (A+B - Lz)_{C}.
\]

The following approximation result can be extracted from \cite[Lemma 7.4]{Petrov1}:
\begin{equation}\label{eq:app_P}
P_L(w,x,y) = C_L \left( \frac{(w+x)(w+y+1)}{1+y-x} \right)^{1/2} \exp(LS(w;x,y) + O(1/L))
\end{equation}
where $C_L$ may depend on $A,B,C$ but not on $w,x,y$,
\begin{multline}\label{eq:def_S}
S(w; x, y) = (w +x)\ln(w+x) - (w+y +1)\ln(w+y+1) \\ + (1+y-x)\ln(1+y-x)+  (a-w)\ln(a-w) + (a+b+c-w)\ln(a+b+c-w)\\ - (-w)\ln(-w) - (a+b-w)\ln(a+b-w)
\end{multline}
and
the $O(1/L)$ terms can be taken uniform on the integration contours.

As proven in \cite{Petrov1}, if $(x,y)$ is in the ellipse $\Circle_{abc}$ then
$S$ has two conjugate non-real critical points ${\rm w}_c,\bar{\rm w}_c$
(say with ${\rm w}_c$ in the upper half complex plane). Moreover,
$S''({\rm w}_c;x,y)\ne 0$. 

Recall that
the integration contour for $w$ includes that for $z$. As explained in
\cite{Petrov1}: 
\begin{proposition}
  \label{prop:muovere}
One can move the integration contours so that:
\begin{itemize}
\item they cross exactly at the two points ${\rm w}_c,\bar{\rm w}_c$;
\item the $w$ contour lies in the 
region of the complex plane where $\Re(S(w;x,y)-S({\rm w}_c;x,y))<0$
(except at ${\rm w}_c, \bar{\rm w}_c$ where the real part is obviously
zero);
\item the $z$ contour lies in the 
region of the complex plane where $\Re(S(z;x,y)-S({\rm w}_c;x,y))>0$
(except at ${\rm w}_c, \bar{\rm w}_c$ where  it is zero);
\item the contours avoid any poles of the integrand (that are on the
  real axis).
\end{itemize}

\end{proposition}
See Figure \ref{fig:nuovicontorni}.
\begin{figure}
  \includegraphics[width=7cm]{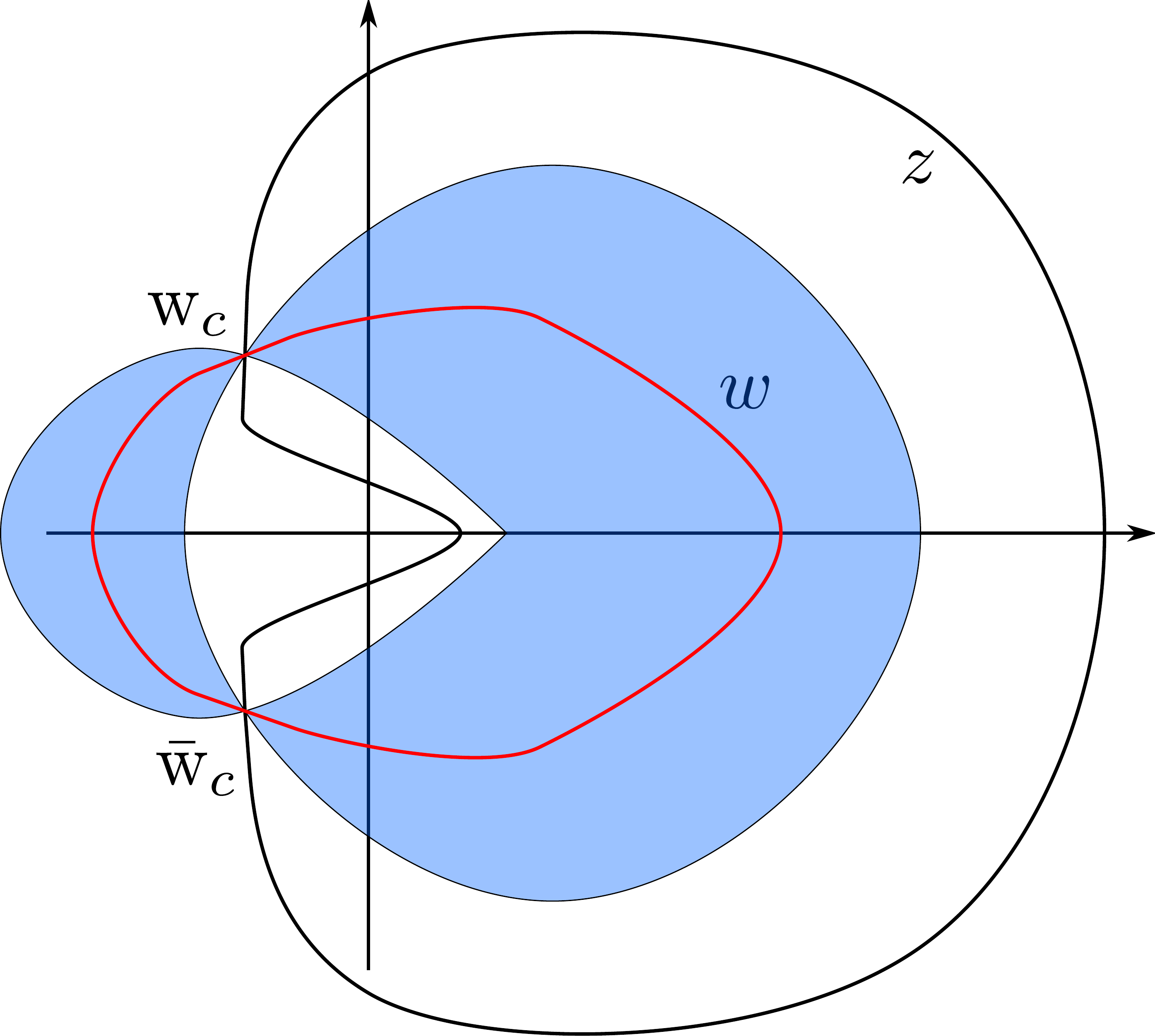}
\caption{A schematic view of the new contours for $z$ and $w$. The shaded region corresponds to $\{w\in\mathbb C :\Re(S(w)-S(\wc))<0\}$. See also \cite[Fig. 11]{Petrov1}.}
\label{fig:nuovicontorni}
\end{figure}
Of course, in the process of moving the contours some residue will
appear, because there is a pole $1/(w-z)$ and the new contours cross. As
shown in \cite[Lemma 7.9]{Petrov1}, the residue
is exactly the single integral
\[
\frac{1}{2\pi i} \int_{\bar{\rm w}_c}^{{\rm w}_c}
  \frac{(1+y-x)}{(z+x)(z+y +1)} \frac{P_L(z, y,x)}{P_L(z,y,x)} dz=\frac{1}{2\pi i} \int_{\bar{\rm w}_c}^{{\rm w}_c}
  \frac{(1+y-x)}{(z+x)(z+y +1)} dz
\]
which is nothing but $
\Pi_\infty(x,y)$.
To prove \eqref{eq:kinf} it remains therefore only to show that the double integral in the
r.h.s. of \eqref{eq:4}, where now $z$ and $w$ run along the new
contours, gives a contribution $O(1/L)$.

\subsection{Integral approximations}

For the double integral we use \eqref{eq:app_P} to get
\begin{multline}
\label{eq:dun}
  \frac{1}{(2i\pi)^2} \oint \oint \left(\frac{(1+y-x)}{(z+x)(z +y +1)} \right)^{\frac{1}{2}}\left(\frac{(1+y-x)}{(w+x)(w +y +1)} \right)^{\frac{1}{2}} \\
    \times \frac{1}{w-z}\exp L[S(w;x,y)-S(z;x,y)] (1+O(1/L)) dz dw
\end{multline}

Recall the choice of the integration contours for $z$ and $w$
described in Proposition \ref{prop:muovere} and observe that along
such contours the exponential in \eqref{eq:dun} is bounded by $1$ in
absolute value. Also, we choose the contours so that close to the
critical points they are exactly linear.

We fix some small $\delta>0$ and we  divide the integral into three regions:
\begin{enumerate}
\item [{\bf Region 1}] $w,z$ are both  within distance
  $L^{-1/2+\delta}$ from ${\rm w}_c$ or both within distance
  $L^{-1/2+\delta}$ from $\bar{\rm w}_c$;
\item [{\bf Region 2}] $z$ is within distance $L^{-1/2+\delta}$ from
  ${\rm w}_c$ and $w$  is within distance $L^{-1/2+\delta}$ from
  $\bar{\rm w}_c$, or viceversa;
\item [{\bf Region 3}] the rest of the integration contours.
\end{enumerate}

Let us consider {\bf Region 3} first, and assume that $z$ is at
distance at least $L^{-1/2+\delta}$ from both critical points.
The first two factors in the integral are bounded
(because the contours stay away from the poles). The factor $1/(w-z)$
can be upper bounded by $L^{1/2-\delta}$ in absolute value. The
exponential is $O(\exp(-L^{2\delta}))$. Indeed, recall that
$\Re(S(w;x,y)-S({\rm w}_c;x,y))\le 0$ while $\Re(S(z;x,y)-S({\rm
  w}_c;x,y))> 0$. More precisely, since $S''({\rm w_c},x,y)\ne 0$ and
the third derivative is finite, one has $\Re(S(z;x,y)-S({\rm
  w}_c;x,y))> L^{-1+2\delta}$ from a Taylor expansion. Finally the $O(1/L)$
term in \eqref{eq:dun} is bounded on the contour so it gives negligible contribution. In conclusion,
the contribution from {\bf Region 3} is $O(\exp(-L^{2\delta}))$.

Now let us consider {\bf Region 2} and assume by symmetry that $w$ is close to ${\rm w}_c$ and $z$ is close to $\bar {\rm w}_c$. We can write $\Re(S(w;x,y)-S({\rm w}_c;x,y))=- C_1 |w-{\rm w}_c|^2+O(|w-{\rm
  w}_c|^3)$, $\Re(S(z;x,y)-S({\rm w}_c;x,y))=-C_2|z-\bar {\rm w}_c|^2+O(|z-\bar {\rm
  w}_c|^3)$ for some $C_i>0$. Furthermore since $z$ and $w$ are far from each other we can use a uniform bound on all terms outside the exponential and them we get a product of two Gaussian integrals, each
of which gives a $O(1/\sqrt L)$ contribution. Again the $O(1/L)$ term does not play any role.

Finally we consider {\bf Region 1} which is more difficult because the
$\frac{1}{w-z}$ term is singular. By symmetry we will only consider the case with
both $w$ and $z$ close to ${\rm w}_c$.
Recall that the integration contours are linear near $\wc$, i.e. one
has
\[
w=\wc+t \theta,\; t\in[-L^{-1/2+\delta},L^{-1/2+\delta}],\quad z=\wc+s \zeta,\; s\in[-L^{-1/2+\delta},L^{-1/2+\delta}]
\]
with $\theta,\zeta$ two modulus-1 complex numbers.

The expression we have to control is then:
\begin{multline}
 E = \int_{\wc-\theta L^{-1/2+\delta}}^{\wc+\theta L^{-1/2+\delta}} dw
 \int_{\wc-\zeta L^{-1/2+\delta}}^{\wc+\zeta L^{-1/2+\delta}} dz
 \left(\frac{(1+y-x)}{(z+x)(z +y +1)} \right)^{\frac{1}{2}}\\\nonumber
\times\left(\frac{(1+y-x)}{(w+x)(w +y +1)} \right)^{\frac{1}{2}} 
    \frac{1+O(1/L)}{w-z} \exp\left\{ L[S(w;x,y)-S(z;x,y)]\right\}.
\end{multline}
The first step is to get rid of the $\left(\frac{(1+y-x)}{(z+x)(z +y
    +1)} \right)^{\frac{1}{2}}\left(\frac{(1+y-x)}{(w+x)(w +y +1)}
\right)^{\frac{1}{2}}$ factor. Remark that it has a non-zero limit
when $w=z=\wc$ and that we can use a Taylor expansion to approximate
it by $a_0+a_1(z-\wc)+a_2(w-\wc)+O(L^{-1+2\delta})$. Plugging this
into the integral we get an expression of the form :
\begin{multline}
\label{eq:a0a1}
 E =  \iint dw dz  
        \frac{a_0+a_1(z-\wc)+a_2(w-\wc)+O(L^{-1+2\delta}) }{w-z} \\\times\exp L[S(w;x,y)-S(z;x,y)]
\end{multline}
where from now on $\iint dw dz$ means $\int_{\wc-\theta L^{-1/2+\delta}}^{\wc+\theta L^{-1/2+\delta}} dw
 \int_{\wc-\zeta L^{-1/2+\delta}}^{\wc+\zeta L^{-1/2+\delta}} dz$. Note that the $O(1/L)$ term has been absorbed into the $O(L^{-1+2\delta})$. 

We start by looking at the most difficult term, the one proportional
to $a_0$, call it $E_0$. Remark that we have :
\[
 \iint dw dz\frac{1}{w-z} \exp \left\{LS''(\wc)[(w-\wc)^2-(z-\wc)^2]\right\} =0.
\]
Indeed the integrand is an odd function of $(z-\wc,w-\wc)$ and we integrate on a symmetric
 domain (the integral is absolutely convergent). We can thus rewrite $E_0$ as
\begin{multline}
  E_{0} =a_0\iint dw dz \frac{1}{w-z}\exp
  \left\{LS''(\wc)[(w-\wc)^2-(z-\wc)^2] \right\}\\
\nonumber
  \times \left(-1+ \exp \left\{L\left[S(w;x,y)-S(z;x,y) - S''(\wc)((w-\wc)^2-(z-\wc)^2)\right] \right\} \right).
\end{multline}
With a Taylor expansion on $\exp LS(w)$ we get
\begin{multline}
  \label{eq:23}
  E_{0} = a_0\iint dw dz \frac{O\left(L\abs{w- \wc}^3+L\abs{z-\wc}^3
    \right)}{w-z}\\
\times\exp \left\{LS''(\wc)[(w-\wc)^2-(z-\wc)^2] \right\}.
\end{multline}
Remark that this expansion is valid because $L\abs{w-\wc}^3= O(L^{-1/2+3\delta}) = o(1)$. We take the absolute value inside the integral and change variables 
$\tilde z = \sqrt L (z -\wc)$ and $\tilde w = \sqrt L (w -\wc)$ to get
\[
  \abs{E_{0}} \leq  \frac{K}{L} \int_{-\theta L^\delta}^{\theta
    L^\delta} d \tilde w \int_{-\zeta L^\delta}^{\zeta L^\delta} d \tilde z\;\frac{ \abs{\tilde z}^3 + \abs{\tilde w}^3}{\abs{\tilde z - \tilde w}} \exp \Re[S''(\wc)(\tilde w^2 -\tilde z^2 )]
\]
and we emphasize that both $S''(\wc)\tilde w^2$ and $-S''(\wc)\tilde
z^2$ have negative real part.
Finally since the linear integration contours for $\tilde z$ and
$\tilde w$ are not colinear,  we can lower bound $\abs{\tilde z -
  \tilde w}$ by a constant times either $\abs{\tilde z}$ or $\abs{\tilde w}$ so we have
\[
  \abs{E_{0}} \leq \frac{K'}{L} \int_{-\theta L^\delta}^{\theta
    L^\delta} d \tilde w \int_{-\zeta L^\delta}^{\zeta L^\delta} d \tilde z (\abs{\tilde z}^2 + \abs{ \tilde{w}}^2)  \exp \Re[S''(\wc)(\tilde w^2 -\tilde z^2 )]
\]
and the Gaussian integral in the right hand side is bounded uniformly
in $L$. 

\smallskip
Let us go back to the integral $E$. The terms $E_1,E_2$ containing $a_1$ or
$a_2$ can be treated similarly to $E_0$.  They are actually easier
since $(w-\wc)$ or $(z-\wc)$ produce a term $L^{-1/2}\tilde w$ or
$L^{-1/2}\tilde z$: since $\tilde w/(\tilde w-\tilde z)$ is bounded,
one gets for $i=1,2$
\[
 \abs{E_{i}}  \leq \frac{K''|a_i|}{L} \int_{-\theta L^\delta}^{\theta
   L^\delta} d \tilde w \int_{-\zeta L^\delta}^{\zeta L^\delta} d \tilde z  \exp\left\{ \Re[S''(\wc)(\tilde w^2 -\tilde z^2 )]\right\}(1+o(1))=O(1/L).
\]
Finally, the term $E_R$ containing the  $O(L^{-1+2\delta})$ error can
be estimated once again with the same Taylor expansion of $S$ and the same
change of variables, and it turns out to be $O(L^{-3/2+2\delta})$.
We just have to remark that the integral
\[
  \int_{-\theta L^\delta}^{\theta L^\delta} d \tilde w \int_{-\zeta
    L^\delta}^{\zeta L^\delta} d \tilde z \frac{1}{\abs{\tilde z-\tilde w}}\exp \Re[S''(\wc)(\tilde w^2 -\tilde z^2 )]
\]
does not diverge with $L$ since 
the singularity  $\frac{1}{\abs{\tilde z-\tilde w}}$ is integrable in $\bbR^2$ (recall that even though $\tilde z$ and $\tilde w$ are complex they live on different lines so we are really integrating on a two dimensional box) and thus the Gaussian integral is bounded.

Overall we have proven that the double integral \eqref{eq:dun} is $O(1/L)$.
 
 \section*{Acknowledgments}
We are very grateful to Alexei Borodin and  Leonid Petrov for
precious indications on the methods of \cite{Petrov1,Petrov2}, and
to Pietro Caputo and Fabio Martinelli for many enlightening discussions.
F. L. T. acknowledges the support of Agence Nationale de la Recherche through grant ANR-2010-BLAN-0108.

\end{document}